\documentclass{aims}
\usepackage{amsmath}
\usepackage{paralist}
\usepackage{graphicx}
\usepackage[colorlinks=true]{hyperref}
\hypersetup{urlcolor=blue, citecolor=red}

\def\currentvolume{X}
\def\currentissue{X}
\def\currentyear{201X}
\def\currentmonth{XX}
\def\ppages{X--XX}
\def\DOI{10.3934/xx.xx.xx.xx}

\newtheorem{theorem}{Theorem}[section]
\newtheorem{lemma}[theorem]{Lemma}
\newtheorem{proposition}{Proposition}
\theoremstyle{definition}
  \newtheorem{definition}[theorem]{Definition}
  \newtheorem{remark}{Remark}
\usepackage{tikz}
\usetikzlibrary{patterns}
\usepackage{mathtools}
\mathtoolsset{showonlyrefs=true}
\usepackage{bm}
\usepackage{mathrsfs}
\usepackage{enumitem}

\title[Wall Effect on the Motion of a Rigid Body]
  {Wall Effect on the Motion of a Rigid Body Immersed in a Free Molecular Flow}

\author[Kai Koike]{}


\email{koike@math.keio.ac.jp}

\thanks{The author would like to express his sincere gratitude to Tatsuo Iguchi for reading the manuscript and giving him valuable comments.}

\begin{document}
\maketitle

\centerline{\scshape Kai Koike}
\medskip
{\footnotesize
  \centerline{Department of Mathematics, Faculty of Science and Technology, Keio University}
  \centerline{3--14--1 Hiyoshi, Kohoku-ku, Yokohama, 223--8522, Japan}
}
\bigskip

\centerline{(Communicated by Mario Pulvirenti)}

\begin{abstract}
  Motion of a rigid body immersed in a semi-infinite expanse of gas in a $d$-dimensional region bounded by an infinite plane wall is studied for free molecular flow on the basis of the free Vlasov equation under the specular boundary condition. We show that the velocity $V(t)$ of the body approaches its terminal velocity $V_{\infty}$ according to a power law $V_{\infty}-V(t)\approx Ct^{-(d-1)}$ by carefully analyzing the pre-collisions due to the presence of the wall. The exponent $d-1$ is smaller than $d+2$ for the case without the wall found in the classical work by Caprino, Marchioro and Pulvirenti~[Comm. Math. Phys., \textbf{264} (2006), pp. 167--189] and thus slower convergence rate results from the presence of the wall.
\end{abstract}

\section{Introduction}
Consider a cylinder immersed in a semi-infinite expanse of gas in a quiescent equilibrium in a $d$-dimensional region bounded by an infinite plane wall (Figure~\ref{fig:cylinder}). We study the motion of this cylinder for free molecular flow on the basis of the free Vlasov equation~\eqref{Vlasov} under the specular boundary condition~\eqref{specular}. The cylinder is accelerated instantaneously to an initial velocity $V_0$ in the direction away from the plane wall and parallel to the axis of the cylinder and thereafter applied a constant force $E$ in this direction. As the cylinder moves through the gas, a drag force $D_V(t)$ is exerted to the body from the surrounding gas and the velocity of the cylinder $V(t)$ is determined dynamically by the balance of the applied force $E$ and the drag force $D_V(t)$: $dV(t)/dt=E-D_V(t)$, where we assumed that the mass of the cylinder is unity.

In the long time limit, the velocity $V(t)$ of the cylinder is expected to approach a terminal velocity $V_{\infty}$ and our primary interest in this paper is the rate of approach to the terminal velocity. The pioneering work by Caprino, Marchioro and Pulvirenti~\cite{CMP1} showed that in the absence of the wall, the asymptotic convergence rate obeys a power law $t^{-(d+2)}$, that is, $V_{\infty}-V(t)\approx Ct^{-(d+2)}$. This is due to the dependence of the drag force on the past history of the motion. Note that if the drag force is solely determined from the instantaneous velocity of the cylinder, that is, $D_V(t)=D(V(t))$ for some smooth and increasing function $D(U)$, the approach to the terminal velocity is exponential in time. Since molecules colliding with the cylinder might had collisions previously, which we call \textit{pre-collisions}, the drag force exerted to the cylinder depends on the whole history of the motion: $\{ V(t)\mid t\geq 0 \}$.

We show in this paper that the presence of the wall modifies the asymptotic convergence rate to $t^{-(d-1)}$, which is slower than the case without the wall. This is caused by the pre-collisions of the molecules with large horizontal velocities $\xi_1$, satisfying in particular $\xi_1>\sup \{ V(t)\mid t\geq 0 \}$. Note that there are no such pre-collisions in the absence of the wall. Our proof is based on the framework developed in~\cite{CMP1} and the novelty of this work lies in establishing appropriate estimates for the pre-collisions due to the presence of the wall. Our results show that a distant obstacle, the plane wall in our case, may change the asymptotic behaviour substantially. This might be particularly important in the design and interpretation of results of laboratory experiments because vacuum chamber has a wall.

We consider in this paper both cases of $0<V_0<V_{\infty}$ and $0<V_{\infty}<V_0$. In the absence of the wall, the case of $0<V_0<V_{\infty}$ is treated in~\cite{CMP1} and the case of $0<V_{\infty}<V_0$ is treated in~\cite{CCM1}. Note that the more difficult case of $0=V_{\infty}<V_0$ is also treated in~\cite{CCM1} but we still are not able to treat this case in the presence of the wall. As in the case without the wall, there is a change in sign of $V_{\infty}-V(t)$ in the case of $0<V_{\infty}<V_0$. We note that in the absence of the wall, an almost necessary and sufficient condition for velocity reversal is given in~\cite{CS2}. Body shapes other than a cylinder can be considered as well. In the absence of the wall, general convex bodies are treated in~\cite{Cavallaro1} and the asymptotic convergence rate is shown to be $t^{-(d+2)}$. Note that due to the presence of the wall, even if the body is convex the asymptotic convergence rate may be different from $t^{-(d+2)}$. In the presence of the wall, general convex bodies or a U-shaped body treated in~\cite{RS1} may be included but we shall restrict ourselves to the case of a cylinder for simplicity. We note that for the model considered in this paper, the drag force $D_V(t)$ is time dependent even when the velocity $V(t)$ is constant. This feature is seen also in the case of a V-shaped body analyzed in~\cite{FSS1}. For they only treated the time dependency of the drag force when the velocity $V(t)$ is constant, our result is the first case to obtain a precise asymptotic convergence rate for a model with such feature. Other boundary conditions, for example the Maxwell boundary condition, are treated in~\cite{ACMP1,CS1} for the case without the wall. The cases with the wall and boundary conditions other than the specular one are left for future study. For other related mathematical and numerical studies, we refer the book by Butt\`{a}, Cavallaro and Marchioro~\cite{BCM1} and the references therein. We briefly mention here that the approach to the terminal velocity is algebraic also for a (non-stationary) Stokes fluid~\cite{BJJ1,CM1,CMT1}. This is also caused by the dependence of the drag force on the past history of the motion.

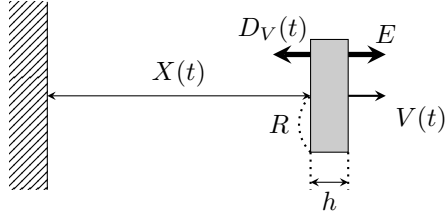
\begin{figure}[htbp]
  \centering
  \begin{tikzpicture}
    \fill[pattern=north east lines] (0,0) -- (0.5,0) -- (0.5,2.5) -- (0,2.5) -- (0,0);
    \draw (0.5,0) -- (0.5,2.5);
    \draw[<->,>=stealth] (0.5,1.25) -- node[above] {$X(t)$} (4,1.25);
    \filldraw[fill=black!20,draw=black] (4,0.5) rectangle (4.5,2);
    \draw[->,>=stealth,black,thick] (4.5,1.25) -- (5,1.25) node[below right] {$V(t)$};
    \draw[->,>=stealth,black,line width=2] (4.5,1.8) -- (5,1.8) node[above] {$E$};
    \draw[->,>=stealth,black,line width=2] (4,1.8) -- (3.5,1.8) node[above] {$D_V(t)$};
    \draw[thick,dotted] (4,0.5) -- (4,0);
    \draw[thick,dotted] (4.5,0.5) -- (4.5,0);
    \draw[<->,>=stealth] (4,0.1) -- node[below] {$h$} (4.5,0.1);
    \draw[thick,dotted] (4,0.5) to [out=135,in=225] node[left] {$R$} (4,1.25);
  \end{tikzpicture}
  \caption{A two dimensional picture of a cylinder immersed in a semi-infinite expanse of gas in a region bounded by an infinite plane wall is shown. A constant force $E$ is applied in the direction of the axis of the cylinder and a drag force $D_V(t)$ is exerted to the cylinder from the surrounding gas. The balance between the applied force $E$ and the drag force $D_V(t)$ determines the dynamics of the cylinder: $dV(t)/dt=E-D_V(t)$.}
  \label{fig:cylinder}
\end{figure}

\section{Mathematical Formulation of the Problem}\label{sec2}
Consider a cylinder of radius $R$ and height $h$ in a $d$-dimensional Euclidean space $\mathbb{R}^d$. We assume that $d\geq 2$. Writing the coordinate of the space as $x=(x_1,x_{\perp})\in \mathbb{R} \times \mathbb{R}^{d-1}$, this cylinder is described by the set
\begin{equation}
  C_V(t)=\{ x=(x_1,x_{\perp}) \mid X(t)<x_1<X(t)+h,\, |x_{\perp}|<R \}.
\end{equation}
$X(t)$ is the $x_1$-coordinate of the left end of the cylinder and is allowed to vary in time (Figure~\ref{fig:cylinder}). This cylinder is placed in a $d$-dimensional region bounded by an infinite plane wall which is the half plane $\mathbb{R}_{+}^{d}=\{ x=(x_1,x_{\perp}) \mid x_1>0 \}$. The lateral boundary of $C_V(t)$ is denoted $C_{V}^{S}(t)$, that is,
\begin{equation}
  C_{V}^{S}(t)=\{ x=(x_1,x_{\perp}) \mid X(t)\leq x_1 \leq X(t)+h,\, |x_{\perp}|=R \}.
\end{equation}
The right and the left boundary of $C_V(t)$ are written $C_{V}^{\pm}(t)$ respectively:
\begin{align}
  C_{V}^{+}(t) & =\{ x=(x_1,x_{\perp}) \mid x_1=X(t)+h,\, |x_{\perp}|<R \}, \\
  C_{V}^{-}(t) & =\{ x=(x_1,x_{\perp}) \mid x_1=X(t),\, |x_{\perp}|<R \}.
\end{align}
Note that when we use $V$ as a subscript, this means that the subscripted quantity may depend on the whole history of the motion, that is, $\{ V(t)\mid t\geq 0\}$.

The state of the gas is described by the distribution function $f(x,\xi,t)$, where $x$ is the position variable and moves through $\overline{\mathbb{R}_{+}^{d}}\backslash C_V(t)$ and $\xi=(\xi_1,\xi_{\perp}) \in \mathbb{R} \times \mathbb{R}^{d-1}$ is the velocity variable of the constituent molecules. We assume that the region $\overline{\mathbb{R}_{+}^{d}}\backslash C_V(t)$ is initially occupied by a semi-infinite expanse of ideal monatomic gas in a quiescent equilibrium of pressure $p_0$ and temperature $T_0$, that is, the distribution function $f$ is initially the Maxwell--Boltzmann distribution with pressure $p_0$ and temperature $T_0$.
\begin{equation}\label{Maxwell}
  f(x,\xi,0)=f_0(\xi)=\frac{2\beta_{0}^{(d+2)/2}}{\pi^{d/2}}p_0\exp\left( -\beta_0 |\xi|^2 \right),
\end{equation}
where $\beta_0=(2R_g T_0)^{-1}$ and $R_g$ is the gas constant. We consider the gas to be in the free molecular regime. This means that the time evolution of the distribution function $f$ obeys the free Vlasov equation
\begin{equation}\label{Vlasov}
  \partial_t f+\xi \cdot \nabla_{x}f=0 \quad \text{in $\Bigl( \mathbb{R}_{+}^{d}\backslash \overline{C_V(t)} \Bigr) \times \mathbb{R}^{d}\times (0,\infty)$}.
\end{equation}
For the boundary condition, we impose the specular boundary condition.
\begin{equation}\label{specular}
  \begin{dcases}
    f(x,\xi,t)=f(x,\xi-2[(\xi-\mathbf{V})\cdot \mathbf{n}]\mathbf{n},t) & \text{for $x\in \partial C_V(t)$ and $(\xi-\mathbf{V})\cdot \mathbf{n}>0$}, \\
    f(x,\xi,t)=f(x,(-\xi_1,\xi_{\perp}),t) & \text{for $x\in \partial \mathbb{R}_{+}^{d}$ and $\xi_1>0$}.
  \end{dcases}
\end{equation}
Here $\mathbf{V}=(V(t),\bm{0})\in \mathbb{R} \times \mathbb{R}^{d-1}$ and $\mathbf{n}$ is the outward unit normal vector to $\partial C_V(t)$.

A constant force $E>0$ is applied to the cylinder in the direction away from the plane wall and parallel to the axis of the cylinder. A drag force $D_V(t)$ is exerted to the body from the surrounding gas and is given by
\begin{align}\label{drag}
  \begin{aligned}
  D_V(t)
  & =2\left[ \int_{C_{V}^{+}(t)}\, dS\int_{\xi_1<V(t)}(\xi_1-V(t))^2 f(x,\xi,t)\, d\xi \right. \\
  & \phantom{2\int_{C_{V}^{+}(t)}\, dS}\left. -\int_{C_{V}^{-}(t)}\, dS\int_{\xi_1>V(t)}(\xi_1-V(t))^2 f(x,\xi,t)\, d\xi \right].
  \end{aligned}
\end{align}
See~\cite{CMP1} for the derivation of the formula. Note that there is no contribution to the drag force from the lateral boundary $C_{V}^{S}(t)$ of the cylinder. Therefore the dynamics of the cylinder is described by the equations
\begin{equation}\label{EOM}
  \begin{dcases}
    dX(t)/dt=V(t), \\
    dV(t)/dt=E-D_V(t).
  \end{dcases}
\end{equation}
Here we assumed that the mass of the cylinder is unity. The initial conditions are
\begin{equation}\label{ICforEOM}
  X(0)=L, \quad V(0)=V_0,
\end{equation}
where $L$ and $V_0$ are positive constants.

For a fixed velocity $\xi$, the Vlasov equation~\eqref{Vlasov} is a constant coefficient transport equation and is therefore solvable by the method of characteristics. We write the set of $(x,\xi)$ representing incoming molecules into $C_{V}^{+}(t)$ and $C_{V}^{-}(t)$ by
\begin{align}
  I_{V}^{+}(t) & =\{ (x,\xi) \mid x\in C_{V}^{+}(t),\, \xi_1<V(t) \}, \\
  I_{V}^{-}(t) & =\{ (x,\xi) \mid x\in C_{V}^{-}(t),\, \xi_1>V(t) \}
\end{align}
respectively. For $(x,\xi)\in I_{V}^{+}(t)\cup I_{V}^{-}(t)$, we denote the backward characteristics starting from $(x,\xi)$ by $(x(s;x,\xi,t),\xi(s;x,\xi,t))$, where $0\leq s\leq t$. We may hereafter write $x(s)=x(s;x,\xi,t)$ and $\xi(s)=\xi(s;x,\xi,t)$ for notational convenience. More precisely, the definition of the characteristics $(x(s),\xi(s))$ are given as follows. Define the first pre-collision time $\tau_1=\tau_1(x,\xi,t)$ by
\begin{align}\label{tau1w}
  \begin{aligned}
  \tau_1
  & =\sup \Bigl\{ s\in [0,t) \mathrel{\Big|} \text{$x-(t-s)\xi \in C_{V}^{+}(s)\cup C_{V}^{-}(s)$} \\
  & \phantom{=\sup  s\in [0,t) \mathrel{\Big|} \text{$x-(t-s)\xi$}} \text{or $x_1-(t-s)\xi_1=0$} \Bigr\} \vee 0,
  \end{aligned}
\end{align}
where $x\vee y=\max(x,y)$. Here we use the convention that the supremum of the empty set equals $-\infty$. For $\tau_1 \leq s\leq t$, the characteristics $(x(s),\xi(s))$ are defined by
\begin{align}
  \begin{aligned}
    & x(s)=x-(t-s)\xi, \\
    & \xi(s)=\xi.
  \end{aligned}
\end{align}
If $\tau_1=0$, we have thus defined the characteristics $(x(s),\xi(s))$ for $0\leq s\leq t$. If $\tau_1>0$, we define the reflected velocity $\xi'(\tau_1)$ as follows.
\begin{equation}
  \begin{dcases}
    \xi_{1}'(\tau_1)=2V(\tau_1)-\xi_1 & \text{if $x(\tau_1)\in C_{V}^{+}(\tau_1)\cup C_{V}^{-}(\tau_1)$}, \\
    \xi_{1}'(\tau_1)=-\xi_1 & \text{if $x_1(\tau_1)=0$}
  \end{dcases}
\end{equation}
and $\xi_{\perp}'(\tau_1)=\xi_{\perp}$. Next define the second pre-collision time $\tau_2$ by
\begin{align}
  \tau_2
  & =\sup \Bigl\{ s\in [0,\tau_1) \mathrel{\Big|} \text{$x(\tau_1)-(\tau_1-s)\xi'(\tau_1)\in C_{V}^{+}(s)\cup C_{V}^{-}(s)$} \\
  & \phantom{s\in [0,\tau_1) \mathrel{\Big|} \text{$x(\tau_1)-\xi'(\tau_1)(\tau_1-s)$}} \text{or $x_1(\tau_1)-(\tau_1-s)\xi_1'(\tau_1)=0$} \Bigr\} \vee 0.
\end{align}
For $\tau_2\leq s<\tau_1$, the characteristics $(x(s),\xi(s))$ are defined by
\begin{align}
  \begin{aligned}
    & x(s)=x(\tau_1)-(\tau_1-s)\xi'(\tau_1), \\
    & \xi(s)=\xi'(\tau_1).
  \end{aligned}
\end{align}
If $\tau_2=0$, we have thus defined the characteristics $(x(s),\xi(s))$ for $0\leq s\leq t$. In general, if $\tau_{n-1}>0$ for $n\geq 2$, we define the reflected velocity $\xi'(\tau_{n-1})$ by
\begin{equation}
  \begin{dcases}
    \xi_{1}'(\tau_{n-1})=2V(\tau_{n-1})-\xi_1(\tau_{n-1}) & \text{if $x(\tau_{n-1})\in C_{V}^{+}(\tau_{n-1})\cup C_{V}^{-}(\tau_{n-1})$}, \\
    \xi_{1}'(\tau_{n-1})=-\xi_1(\tau_{n-1}) & \text{if $x_1(\tau_{n-1})=0$}
  \end{dcases}
\end{equation}
and $\xi_{\perp}'(\tau_{n-1})=\xi_{\perp}$. The $n$-th pre-collision time $\tau_n$ is defined by
\begin{align}\label{taunw}
  \begin{aligned}
  \tau_n
  & = \sup \Bigl\{ s\in [0,\tau_{n-1}) \mathrel{\Big|} \text{$x(\tau_{n-1})-(\tau_{n-1}-s)\xi'(\tau_{n-1})\in C_{V}^{+}(s)\cup C_{V}^{-}(s)$} \\
  & \phantom{s\in [0,\tau_{n-1}) \mathrel{\Big|} \text{$x(\tau_{n-1},t)$}} \text{or $x_1(\tau_{n-1})-(\tau_{n-1}-s)\xi_1'(\tau_{n-1})=0$} \Bigr\} \vee 0.
  \end{aligned}
\end{align}
For $\tau_n \leq s<\tau_{n-1}$, the characteristics $(x(s),\xi(s))$ are defined by
\begin{align}
  \begin{aligned}
    & x(s)=x(\tau_{n-1})-(\tau_{n-1}-s)\xi'(\tau_{n-1}), \\
    & \xi(s)=\xi'(\tau_{n-1}).
  \end{aligned}
\end{align}

Similar to~\cite[Proposition~A.1]{CMP1}, we can prove that for $(x,\xi)\in I_{V}^{+}(t)\cup I_{V}^{-}(t)$, we have $\xi_1(\tau_n)\neq V(\tau_n)$ and $\tau_{m}=0$ for some $m\geq 1$ except on a set of zero $(2d-1)$-dimensional Hausdorff measure in $I_{V}^{+}(t)\cup I_{V}^{-}(t)$. Note that we also have $\xi_1(\tau_n)\geq V(\tau_n)$ (resp. $\xi_1(\tau_n)\leq V(\tau_n)$) for $(x,\xi)\in I_{V}^{+}(t)$ (resp. $(x,\xi)\in I_{V}^{-}(t)$), which is intuitively clear and is proved in the proof of Proposition~\ref{r+nonnegative}. Therefore we see that $\tau_n<\tau_{n-1}$ as long as $\tau_{n-1}>0$. Moreover, the case of $|x_{\perp}(\tau_n)|=R$ is also measure theoretically negligible. From these, we see that the characteristics $(x(s),\xi(s))$ are well-defined up to $s=0$ except on a set of measure zero.

Note that the specular boundary condition~\eqref{specular} implies that
\begin{equation}
  f(x,\xi,t)=f(x(s),\xi(s),s)
\end{equation}
for $0\leq s\leq t$ and if we denote $\xi_0=(\xi_{01},\xi_{0\perp})=\xi(0)$, we have
\begin{equation}
  f(x,\xi,t)=f_0(\xi_0).
\end{equation}
Therefore we can rewrite equation~\eqref{drag} as
\begin{equation}
  D_V(t)=2\left[ \int_{I_{V}^{+}(t)}(\xi_1-V(t))^2 f_0(\xi_0)\, d\xi dS-\int_{I_{V}^{-}(t)}(\xi_1-V(t))^2 f_0(\xi_0)\, d\xi dS \right].
\end{equation}
We split $D_V(t)$ into three parts: $D_V(t)=D_0(V(t))+r_{V}^{+}(t)+r_{V}^{-}(t)$. First, $D_0(V(t))$ is defined as follows.
\begin{align}\label{D0}
  \begin{aligned}
    & D_0(V(t)) \\
    & \quad =2\left[ \int_{I_{V}^{+}(t)}(\xi_1-V(t))^2 f_0(\xi)\, d\xi dS-\int_{I_{V}^{-}(t)}(\xi_1-V(t))^2 f_0(\xi)\, d\xi dS \right] \\
    & \quad =C_0\left[ \int_{-\infty}^{V(t)}(\xi_1-V(t))^2 e^{-\beta_0 \xi_{1}^{2}}\, d\xi_1-\int_{V(t)}^{\infty}(\xi_1-V(t))^2 e^{-\beta_0 \xi_{1}^{2}}\, d\xi_1 \right],
  \end{aligned}
\end{align}
where $C_0=4\pi^{(d-2)/2}\beta_{0}^{3/2}R^{d-1}p_0 /\Gamma((d+1)/2)$. Secondly, $r_{V}^{\pm}(t)$ is defined as follows.
\begin{align}
  r_{V}^{+}(t) & =\bar{C}\int_{I_{V}^{+}(t)}(\xi_1-V(t))^2 \left( e^{-\beta_0 |\xi_{0}|^{2}}-e^{-\beta_0 |\xi|^2} \right) \, d\xi dS, \label{r+} \\
  r_{V}^{-}(t) & =\bar{C}\int_{I_{V}^{-}(t)}(\xi_1-V(t))^2 \left( e^{-\beta_0 |\xi|^2}-e^{-\beta_0 |\xi_{0}|^{2}} \right) \, d\xi dS, \label{r-}
\end{align}
where $\bar{C}=4\beta_{0}^{(d+2)/2}p_0 /\pi^{d/2}$. The following lemma is proved in~\cite{CMP1}.

\begin{lemma}\label{D0properties}
  Consider a function
  \begin{align}
    D_0(U)
    & =C_0\left[ \int_{-\infty}^{U}(\xi_1-U)^2 e^{-\beta_0 \xi_{1}^{2}}\, d\xi_1-\int_{U}^{\infty}(\xi_1-U)^2 e^{-\beta_0 \xi_{1}^{2}}\, d\xi_1 \right]
  \end{align}
  defined for $U\geq 0$. Then $D_0(U)$ is positive and convex for $U>0$. Moreover, there is a constant $C_1>0$ such that $D_{0}'(U)\geq C_1$ for $U\geq 0$. In particular, $D_0(U)$ grows at least linearly in $U$.
\end{lemma}

By Lemma~\ref{D0properties}, the equation
\begin{equation}\label{terminal}
  D_0(V_{\infty})=E
\end{equation}
is uniquely solvable for $V_{\infty}$ given $E>0$. Therefore we can and will consider $V_{\infty}$ to be the parameter of the problem instead of $E$. We prove in the sequel that $V_{\infty}$ is the terminal velocity, that is, $\lim_{t\to \infty}V(t)=V_{\infty}$.

\section{Main Theorems}
From now on, we say that $(X(t),V(t))$ is a solution to the problem if $V(t)$ is Lipschitz continuous and $(X(t),V(t))$ satisfies equations~\eqref{EOM} and~\eqref{ICforEOM} with $D_V(t)=D_0(V(t))+r_{V}^{+}(t)+r_{V}^{-}(t)$, where $D_0(V(t))$ and $r_{V}^{\pm}(t)$ are defined by equations~\eqref{D0},~\eqref{r+} and~\eqref{r-}.

\begin{theorem}\label{theorem1}
  Let $0<V_0<V_{\infty}$ and $\gamma=V_{\infty}-V_0$. There exist positive constants $\gamma_0=\gamma_0(\beta_0,p_0,R,V_{\infty},d)$ and $L_0=L_0(\beta_0,p_0,R,V_{\infty},d)$ such that for all $\gamma \in (0,\gamma_0]$ and $L\in [L_0,\infty)$, there exists a solution $(X(t),V(t))$ to the problem. Moreover, for all $\gamma \in (0,\gamma_0]$ and $L\in [L_0,\infty)$, any solution $(X(t),V(t))$ to the problem satisfies the following inequalities: $V(t)\geq V_{\infty}/2$,
  \begin{equation}
    V_{\infty}-V(t)\leq \gamma e^{-C_{+}t}+\gamma^3 \frac{A_+}{(1+t)^{d+2}}+L^{-(d-1)}\frac{B_+}{(1+t/L)^{d-1}} \label{estimate1}
  \end{equation}
  and
  \begin{equation}
    V_{\infty}-V(t)\geq \gamma e^{-C_{-}t} \label{estimate2}
  \end{equation}
  for $t\geq 0$, where $A_+=A_+(\beta_0,p_0,R,V_{\infty},d)$ and $B_+=B_+(\beta_0,p_0,R,V_{\infty},d)$ are positive constants. Here $C_{+}=D_{0}'(V_{\infty}/2)$ and $C_{-}=D_{0}'(V_{\infty})$.
\end{theorem}

\begin{theorem}\label{theorem2}
  Let $0<V_0<V_{\infty}$ and $\gamma=V_{\infty}-V_0$. There exist positive constants $\gamma_0=\gamma_0(\beta_0,p_0,R,V_{\infty},d)$ and $L_0=L_0(\beta_0,p_0,R,V_{\infty},d)$ such that for all $\gamma \in (0,\gamma_0]$ and $L\in [L_0,\infty)$, any solution $(X(t),V(t))$ to the problem satisfies the inequality
  \begin{equation}
    V_{\infty}-V(t)\geq \gamma e^{-C_{-}t}+\bm{1}_{\{ t>L \}}\frac{B_-}{t^{d-1}} \label{estimate3}
  \end{equation}
  for $t\geq 0$, where $B_-=B_-(\beta_0,p_0,R,V_{\infty},d)$ is a positive constant.
\end{theorem}

\begin{remark}\label{remark1}
  Uniqueness of the solution is not known as in the previous works (e.g. \cite{CMP1}). However, Theorems~\ref{theorem1} and \ref{theorem2} at least guarantee that the asymptotic behaviour is unique; the asymptotic convergence rate is $t^{-(d-1)}$.
\end{remark}

\begin{remark}
  The asymptotic convergence rate is $t^{-(d-1)}$ in our case --- the case with the plane wall; it is $t^{-(d+2)}$ without the plane wall~\cite{CMP1}. So the presence of the wall delays the convergence to the terminal velocity $V_{\infty}$. This is because the presence of the plane wall strengthens the drag force $D_V(t)$.
  
  An explanation for this is as follows. Consider a molecule with velocity $\xi$ impinging on the left side of the cylinder at time $t$: $(x,\xi)\in I_{V}^{-}(t)$. Suppose for simplicity that $\xi_1>V_{\infty}$. If there is no plane wall, then no pre-collisions occur because $V(t)<V_{\infty}$ from inequality~\eqref{estimate2}. So we have $\xi_{0}=\xi$ in this case. On the other hand, if the plane wall is present, then the molecule might have several pre-collisions. Let us assume for simplicity that there are only two pre-collisions: $\tau_2>0$ and $\tau_3=0$. Since $V(t)<V_{\infty}$, the first pre-collision is with the plane wall (at $s=\tau_1$); and the second pre-collision is with the cylinder (at $s=\tau_2$). See Figure~\ref{fig:wallprecoll}. In this case, we have $\xi_{01}=2V(\tau_2)-\xi_{1}'(\tau_1)=2V(\tau_2)+\xi_1$ and $|\xi_{0\perp}|=|\xi_{\perp}|$. So $|\xi_0|$ is larger in the presence of the plane wall. Remember that $f(x,\xi,t)=f_0(\xi_0)$ and $f_0(\xi_0)$ is decreasing in $|\xi_0|$. Therefore, $f(x,\xi,t)$ is smaller if the plane wall is present: The momentum transfer from the surrounding gas to the left side of the cylinder is smaller. This means that the drag force $D_V(t)$ is strengthened by the presence of the plane wall.
\end{remark}
  
\begin{figure}[h]
  \centering
  \includegraphics[scale=0.5]{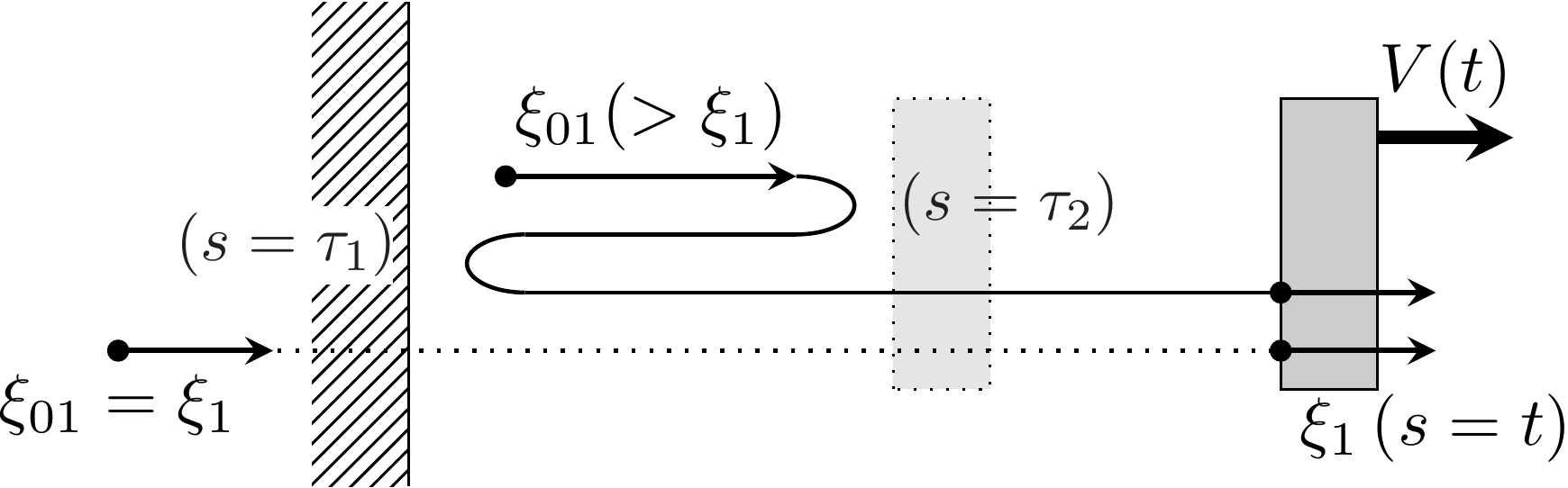}
  \caption{A molecule with velocity $\xi$ is impinging on the left side of the cylinder. The solid line schematically represents the motion of the molecule in the case with the plane wall; the dashed line is that without the plane wall.}
  \label{fig:wallprecoll}
\end{figure}

\begin{remark}\label{remark2}
  Let $(X(t),V(t))$ be any solution to the problem. If we further assume that $L^{-(d-1)}\ll \gamma$,\footnote{The precise meaning of the notation $L^{-(d-1)}\ll \gamma$ is explained in the appendix (Section~\ref{appendix}).} then $V(t)>V_0$ for $t>0$ --- and not just $V(t)\geq V_{\infty}/2$. Under the same assumption, $V(t)$ is also increasing on a time interval $[0,t_0]$, where
  \begin{equation}
    t_0=\frac{1}{2C_-}\log \frac{C_+ \gamma}{\hat{C}\left[ \gamma^3+L^{-(d-1)} \right] }
  \end{equation}
  and $\hat{C}=\hat{C}(\beta_0,p_0,R,V_{\infty},d)$ is a positive constant. Note that $t_0$ grows infinitely as $\gamma \to +0$ since $L^{-(d-1)}\ll \gamma$. These are proved in the appendix (Sections~\ref{appendix1} and \ref{appendix2}).
\end{remark}

\begin{remark}\label{remark3}
  Under an additional assumption that $L^{-(d-1)}\ll \gamma$, we can refine inequality~\eqref{estimate3} to
  \begin{equation}
    V_{\infty}-V(t)\geq \gamma e^{-C_{-}t}+\bm{1}_{\{ t_1>t>\bar{t} \}}\gamma^4 \frac{A_-}{t^{d+2}}+\bm{1}_{\{ t>L \}}\frac{B_-}{t^{d-1}} \label{estimate3_refined},
  \end{equation}
  where $A_-=A_-(\beta_0,p_0,R,V_{\infty},d)$, $\bar{t}=\bar{t}(\beta_0,p_0,R,V_{\infty},d)$ are positive constants and
  \begin{equation}
    t_1=\bar{c}\gamma L^{d-1},
  \end{equation}
  where $\bar{c}=\bar{c}(\beta_0,p_0,R,V_{\infty},d)$ is a positive constant. This is also proved in the appendix (Section~\ref{appendix3}). If we fix $\gamma$ and let $L\to \infty$, inequalities~\eqref{estimate1} and \eqref{estimate3_refined} become
  \begin{equation}
    \gamma e^{-C_{-}t}+\bm{1}_{\{ t>\bar{t} \}}\gamma^4 \frac{A_-}{t^{d+2}}\leq V_{\infty}-V(t)\leq \gamma e^{-C_{+}t}+\gamma^3 \frac{A_+}{(1+t)^{d+2}}.
  \end{equation}
  These are exactly the same estimates obtained in the case without the plane wall~\cite{CMP1}.
\end{remark}

\begin{remark}
  $V(t)$ is necessarily continuously differentiable if $(X(t),V(t))$ is a solution to the problem. See Remark~\ref{remark:cont} after Proposition~\ref{mapcont}.
\end{remark}

We can also prove the following theorems.

\begin{theorem}\label{theorem3}
  Let $0<V_{\infty}<V_0$ and $\gamma=V_0-V_{\infty}$. There exist positive constants $\gamma_0=\gamma_0(\beta_0,p_0,R,V_{\infty},d)$ and $L_0=L_0(\beta_0,p_0,R,V_{\infty},d)$ such that for all $\gamma \in (0,\gamma_0]$ and $L\in [L_0,\infty)$, there exists a solution $(X(t),V(t))$ to the problem. Moreover, any solution $(X(t),V(t))$ to the problem satisfies the following inequalities: $V(t)\geq V_{\infty}/2$,
  \begin{equation}
    V(t)-V_{\infty}\geq \gamma e^{-C_{1}t}-\gamma^3 \frac{A_1}{(1+t)^{d+2}}-L^{-(d-1)}\frac{B_1}{(1+t/L)^{d-1}} \label{estimate4}
  \end{equation}
  and
  \begin{equation}
    V(t)-V_{\infty}\leq \gamma e^{-C_{2}t} \label{estimate5}
  \end{equation}
  for $t\geq 0$, where $A_1=A_1(\beta_0,p_0,R,V_{\infty},d)$ and $B_1=B_1(\beta_0,p_0,R,V_{\infty},d)$ are positive constants. Here $C_{1}=D_{0}'(3V_{\infty}/2)$ and $C_{2}=D_{0}'(V_{\infty}/2)$.
\end{theorem}

\begin{theorem}\label{theorem4}
  Let $0<V_{\infty}<V_0$ and $\gamma=V_0-V_{\infty}$. There exist positive constants $\gamma_0=\gamma_0(\beta_0,p_0,R,V_{\infty},d)$ and $L_0=L_0(\beta_0,p_0,R,V_{\infty},d)$ such that for all $\gamma \in (0,\gamma_0]$ and $L\in [L_0,\infty)$, any solution $(X(t),V(t))$ to the problem satisfies the inequality
  \begin{equation}
    V(t)-V_{\infty}\leq \gamma e^{-C_{2}t}-\bm{1}_{\{ t>L \}}\frac{B_2}{t^{d-1}} \label{estimate6}
  \end{equation}
  for $t\geq 0$, where $B_2=B_2(\beta_0,p_0,R,V_{\infty},d)$ is a positive constant.
\end{theorem}

\begin{remark}
  It follows from inequality~\eqref{estimate6} that $V(t)-V_{\infty}$ changes its sign: $V(t)-V_{\infty}>0$ for $t\ll 1$ and $V(t)-V_{\infty}<0$ for $t\gg 1$.
\end{remark}

We only prove Theorems~\ref{theorem1} and~\ref{theorem2}. Theorems~\ref{theorem3} and~\ref{theorem4} can be proved similarly by the method developed in this paper.

\section{Proof of the Theorems}
\subsection{Strategy of the Proof}
We first briefly describe the strategy of our proof. First, we define a function space $\Omega(\gamma,L,A_*,B_*,V_{\infty})$ as follows. In the following, $A_*$ and $B_*$ are positive constants.

\begin{definition}
  A function $W\colon [0,\infty)\to [V_{\infty}/2,V_{\infty})$ belongs to $\Omega(\gamma,L,A_*,B_*,V_{\infty})$ if $W(t)$ is Lipschitz continuous in $t$, $W(0)=V_0$ and satisfies for $t\geq 0$ the following inequalities.
  \begin{align}
    & V_{\infty}-W(t)\leq \gamma e^{-C_{+}t}+\gamma^3 \frac{A_*}{(1+t)^{d+2}}+L^{-(d-1)}\frac{B_*}{(1+t/L)^{d-1}}, \label{upperbound} \\
    & V_{\infty}-W(t)\geq \gamma e^{-C_{-}t}, \label{lowerbound1}
  \end{align}
  where $C_+=D_{0}'(V_{\infty}/2)$ and $C_-=D_{0}'(V_{\infty})$.
\end{definition}

Next, we define a map $W\mapsto V_W$ by the equations
\begin{equation}\label{map}
  \begin{dcases}
    dX_W(t)/dt=V_W(t), \\
    \frac{d}{dt}\left[ V_{\infty}-V_W(t) \right]=-K_W(t)\left[ V_{\infty}-V_W(t) \right] +r_{W}^{+}(t)+r_{W}^{-}(t)
  \end{dcases}
\end{equation}
with the initial conditions
\begin{equation}\label{mapIC}
  X_W(0)=L,\quad V_W(0)=V_0.
\end{equation}
Here $K_W(t)$ is defined by
\begin{equation}\label{KW}
  K_W(t)=\frac{D_0(V_{\infty})-D_0(W(t))}{V_{\infty}-W(t)}.
\end{equation}
By Lemma~\ref{D0properties}, we have
\begin{equation}\label{KWestimate}
  C_+\leq K_W(t)\leq C_{-}.
\end{equation}
We note here that $r_{W}^{\pm}(t)$ are computed via the characteristics $(x(s),\xi(s))$ determined from the dynamics of the cylinder described by $(X_W(t),W(t))$.

We prove that $V_W \in \Omega=\Omega(\gamma,L,A_+,B_+,V_{\infty})$ if $W\in \Omega$, upon taking $\gamma$ sufficiently small and $L$ sufficiently large (Section~\ref{sec:VW}). Here $A_+=A_+(\beta_0,p_0,R,V_{\infty},d)$ and $B_+=B_+(\beta_0,p_0,R,V_{\infty},d)$ are appropriately chosen positive constants. Since equation~\eqref{map} can be solved explicitly regarding $r_{W}^{\pm}(t)$ as non-homogeneous terms, what we have to do is to obtain suitable estimates for $r_{W}^{\pm}(t)$ (Sections~\ref{sec:r+} and \ref{sec:r-}). The estimates are obtained by carefully analyzing the characteristics (Sections~\ref{sec:char+} and \ref{sec:char-}). Then we obtain a fixed point $V\in \Omega$ of the map $W\mapsto V_W$ by applying Schauder's fixed point theorem (Section~\ref{sec:proof1}). A fixed point $V\in \Omega$ satisfies equations~\eqref{EOM} and~\eqref{ICforEOM}. Obtaining a lower bound for $r_{V}^{-}(t)$, we can improve the lower bound~\eqref{estimate2} and obtain the lower bound~\eqref{estimate3} to prove Theorem~\ref{theorem2} (Sections~\ref{sec:r-lowerbound} and \ref{sec:proof2}).

\subsection{Analysis of the Characteristics for $(x,\xi)\in I_{W}^{+}(t)$}\label{sec:char+}
Let us take $W$ from $\Omega(\gamma,L,A_*,B_*,V_{\infty})$.\footnote{In the following (except Section~\ref{appendix}), $\tau_n$ is defined using $W$ in place of $V$. See Section~\ref{sec2} for the definition of $\tau_n$.} To obtain an estimate for $r_{W}^{+}(t)$, we analyze the characteristics starting from $(x,\xi)\in I_{W}^{+}(t)$. We define the modified first pre-collision time $\tilde{\tau}_1$ by
\begin{equation}\label{tau1}
  \tilde{\tau}_1=\sup \Bigl\{ s\in [0,t) \mathrel{\Big|} \text{$x-(t-s)\xi \in C_{W}^{+}(s)\cup C_{W}^{-}(s)$} \Bigr\} \vee 0.
\end{equation}
Note the difference between $\tau_1$ (definition~\eqref{tau1w}) and $\tilde{\tau}_1$: $\tau_1$ takes into account the plane wall at $x_1=0$ and $\tilde{\tau}_1$ do not. If $\tilde{\tau}_1>0$ (and not just $\tau_1>0$), we have
\begin{equation}
  (t-\tilde{\tau}_1)\xi_1=\int_{\tilde{\tau}_1}^{t}W(s)\, ds.
\end{equation}
See Figure~\ref{fig:tau1}. Introducing a function
\begin{equation}
  \langle W \rangle_{s,t}=\frac{1}{t-s}\int_{s}^{t}W(\sigma)\, d\sigma \quad (0\leq s<t),
\end{equation}
we see that this is equivalent to
\begin{equation}\label{relxi1W}
  \xi_1=\langle W \rangle_{\tilde{\tau}_1,t}.
\end{equation}
Furthermore, the following inequality holds.
\begin{equation}\label{xiperp}
  |x_{\perp}-(t-\tilde{\tau}_1)\xi_{\perp}|<R.
\end{equation}
See Figure~\ref{fig:tau1} again.

\begin{figure}[htbp]
  \centering
  \begin{tikzpicture}
    \centering
    \filldraw[fill=black!20,draw=black] (0,0) rectangle (0.5,2);
    \filldraw[fill=black!20,draw=black] (4,0) rectangle (4.5,2);
    \filldraw[fill=black,draw=black] (4.5,0.4) circle (0.05) node[above right] {$x$};
    \filldraw[fill=black,draw=black] (0.5,1.6) circle (0.05) node at (0.5,1.0) {$x-(t-\tilde{\tau}_1)\xi$};
    \draw[thick,->,>=stealth] (4.5,0.4) -- (5.2,0.19) node[right] {$\xi$};
    \draw[thick,->,>=stealth] (0.5,1.6) -- (1.2,1.39);
    \draw[thick,dotted] (1.2,1.39) -- (4.5,0.4);
    \node[below] at (4.5,0) {$s=t$};
    \node[below] at (0,-0.06) {$s=\tilde{\tau}_1$};
    \draw[thick,<->,>=stealth,black] (0.5,0.2) -- node[below] {$(t-\tilde{\tau}_1)\xi_1$} (4.5,0.2);
    \filldraw[thick,->,>=stealth,line width=2,black] (4.5,1.8) -- node[above] {$W(t)$} (5.2,1.8);
    \draw[thick,<->,>=stealth,black] (0.5,1.8) -- node[above] {$\int_{\tilde{\tau}_1}^{t}W(s)\, ds$} (4.5,1.8);
  \end{tikzpicture}
  \caption{A two dimensional picture of a pre-collision at $C_{W}^{+}(\tilde{\tau}_1)$ is shown. The horizontal distance traversed by the cylinder and the characteristic curve $x(s)$ from $\tilde{\tau}_1$ to $t$ coincide.}
  \label{fig:tau1}
\end{figure}
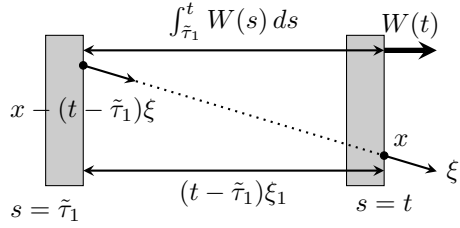

\subsection{Non-Negativity of $r_{W}^{+}(t)$ and its Upper Bound}\label{sec:r+}
We prove here the non-negativity of $r_{W}^{+}(t)$ and obtain its upper bound. First, we derive estimates for $W(t)-\xi_1$. In the following, $C$ represents a positive constant depending only on $\beta_0$, $p_0$, $R$, $V_{\infty}$ and $d$, which might change from place to place.

\begin{lemma}\label{tau1leq}
  Let $W\in \Omega(\gamma,L,A_*,B_*,V_{\infty})$ and $(x,\xi)\in I_{W}^{+}(t)$. If $0<\tilde{\tau}_1\leq t/2$, then
  \begin{equation}
    0<W(t)-\xi_1\leq \frac{C}{1+t}(\gamma+\gamma^3 A_*)+L^{-(d-1)}\frac{B_*}{1+t/L}
  \end{equation}
  for $d\geq 3$ and
  \begin{equation}
    0<W(t)-\xi_1\leq \frac{C}{1+t}(\gamma+\gamma^3 A_*)+L^{-1}B_*
  \end{equation}
  for $d=2$.
\end{lemma}

\begin{remark}
  Note that we are \textit{not} referring to $\tau_1$ here. See definitions~\eqref{tau1w} and \eqref{tau1}.
\end{remark}

\begin{proof}
  By inequality~\eqref{upperbound}, equation~\eqref{relxi1W} and $W(t)<V_{\infty}$, we have for $d\geq 3$
  \begin{align}
    W(t)-\xi_1
    & =W(t)-\langle W \rangle_{\tilde{\tau}_1,t} \\
    & =\frac{1}{t-\tilde{\tau}_1}\int_{\tilde{\tau}_1}^{t}\left[ \left( V_{\infty}-W(s) \right)-\left( V_{\infty}-W(t) \right) \right] \, ds \\
    & \leq \frac{1}{t-\tilde{\tau}_1}\int_{\tilde{\tau}_1}^{t}\left[ \gamma e^{-C_{+}s}+\gamma^3 \frac{A_*}{(1+s)^{d+2}}+L^{-(d-1)}\frac{B_*}{(1+s/L)^{d-1}} \right] \, ds \\
    & \leq \gamma e^{-C_{+}\tilde{\tau}_1}\frac{1-e^{-C_{+}(t-\tilde{\tau}_1)}}{C_{+}(t-\tilde{\tau}_1)} \\
    & \qquad \qquad +\frac{\gamma^3 A_*}{(d+1)(t-\tilde{\tau}_1)} \left[ \frac{1}{(1+\tilde{\tau}_1)^{d+1}}-\frac{1}{(1+t)^{d+1}} \right] \\
    & \qquad \qquad +\frac{L^{-(d-1)}B_*}{(d-2)(t-\tilde{\tau}_1)}\left[ \frac{L}{(1+\tilde{\tau}_1/L)^{d-2}}-\frac{L}{(1+t/L)^{d-2}} \right] \\
    & \leq C\frac{\gamma}{1+t}+\gamma^3 A_* \frac{(1+s_1)^d}{(1+\tilde{\tau}_1)^{d+1}(1+t)^{d+1}} \\
    & \qquad \qquad +L^{-(d-1)}B_* \frac{(1+s_2/L)^{d-3}}{(1+\tilde{\tau}_1)^{d-2}(1+t/L)^{d-2}}
  \end{align}
  for some $s_1$, $s_2 \in (\tilde{\tau}_1,t)$. Therefore we conclude that
  \begin{equation}
    W(t)-\xi_1\leq \frac{C}{1+t}(\gamma+\gamma^3 A_*)+L^{-(d-1)}\frac{B_*}{1+t/L}   .
  \end{equation}
  If $d=2$, we proceed as follows.
  \begin{align}
    W(t)-\xi_1
    & =W(t)-\langle W \rangle_{\tilde{\tau}_1,t} \\
    & =\frac{1}{t-\tilde{\tau}_1}\int_{\tilde{\tau}_1}^{t}\left[ \left( V_{\infty}-W(s) \right)-\left( V_{\infty}-W(t) \right) \right] \, ds \\
    & \leq \frac{1}{t-\tilde{\tau}_1}\int_{\tilde{\tau}_1}^{t}\left[ \gamma e^{-C_{+}s}+\gamma^3 \frac{A_*}{(1+s)^4}+L^{-1}\frac{B_*}{1+s/L} \right] \, ds \\
    & \leq \frac{C}{1+t}(\gamma+\gamma^3 A_*)+L^{-1}\frac{B_*}{t-\tilde{\tau}_1}\log \left( \frac{1+t/L}{1+\tilde{\tau}_1/L} \right)^{L} \\
    & \leq \frac{C}{1+t}(\gamma+\gamma^3 A_*)+L^{-1}\frac{B_*}{t-\tilde{\tau}_1}\log \left[ 1+\frac{(t-\tilde{\tau}_1)/L}{1+\tilde{\tau}_1/L} \right]^L \\
    & \leq \frac{C}{1+t}(\gamma+\gamma^3 A_*)+L^{-1}\frac{B_*}{t-\tilde{\tau}_1}\log \left[ 1+(t-\tilde{\tau}_1)/L \right]^L \\
    & \leq \frac{C}{1+t}(\gamma+\gamma^3 A_*)+L^{-1}B_*.
  \end{align}
\end{proof}

\begin{lemma}\label{tau1geq}
  Let $W\in \Omega(\gamma,L,A_*,B_*,V_{\infty})$ and $(x,\xi)\in I_{W}^{+}(t)$. If $t/2<\tilde{\tau}_1<t$, then
  \begin{equation}
    0<W(t)-\xi_1\leq C\left[ \frac{\gamma+\gamma^3 A_*}{(1+t)^{d+2}}+L^{-(d-1)}\frac{B_*}{(1+t/L)^{d-1}} \right].
  \end{equation}
\end{lemma}

\begin{proof}
  By inequality~\eqref{upperbound}, equation~\eqref{relxi1W} and $W(t)<V_{\infty}$, we have
  \begin{align}
    W(t)-\xi_1
    & =W(t)-\langle W \rangle_{\tilde{\tau}_1,t} \\
    & =\frac{1}{t-\tilde{\tau}_1}\int_{\tilde{\tau}_1}^{t}\left[ \left( V_{\infty}-W(s) \right)-\left( V_{\infty}-W(t) \right) \right] \, ds \\
    & \leq \frac{1}{t-\tilde{\tau}_1}\int_{\tilde{\tau}_1}^{t}\left[ \gamma e^{-C_{+}s}+\gamma^3 \frac{A_*}{(1+s)^{d+2}}+L^{-(d-1)}\frac{B_*}{(1+s/L)^{d-1}} \right] \, ds \\
    & \leq \gamma e^{-C_{+}\tilde{\tau}_1}+\gamma^3 \frac{A_*}{(1+\tilde{\tau}_1)^{d+2}}+L^{-(d-1)}\frac{B_*}{(1+\tilde{\tau}_1/L)^{d-1}} \\
    & \leq C\left[ \frac{\gamma+\gamma^3 A_*}{(1+t)^{d+2}}+L^{-(d-1)}\frac{B_*}{(1+t/L)^{d-1}} \right].
  \end{align}
\end{proof}

Next we prove the non-negativity of $r_{W}^{+}(t)$.

\begin{proposition}\label{r+nonnegative}
  For $W\in \Omega(\gamma,L,A_*,B_*,V_{\infty})$, we have $r_{W}^{+}(t)\geq 0$.
\end{proposition}

\begin{proof}
  By definition~\eqref{r+}, we see that it suffices to prove that $|\xi_{01}|\leq |\xi_1|$ for each $(x,\xi)\in I_{W}^{+}(t)$. Note that we have $\xi_{0\perp}=\xi_{\perp}$. Now we prove that $|\xi_{1}'(\tau_n)|\leq |\xi_1(\tau_n)|$ for $n\geq 1$ if $\tau_n>0$. By the definition of the reflected velocity $\xi'(\tau_n)$, we see that $|\xi_{1}'(\tau_n)|=|\xi_1(\tau_n)|$ if and only if $x_1(\tau_n)=0$; or $x(\tau_n)\in C_{W}^{+}(\tau_n)$ and $\xi_1(\tau_n)=W(\tau_n)$. Note that $\xi_1(\tau_n)=W(\tau_n)$ happens only on a set of measure zero. If $x_1(\tau_n)=0$, we have $\xi_0=\xi'(\tau_n)$ from a simple geometrical reason. Therefore, to prove that $|\xi_{01}|\leq |\xi_1|$, we can assume that $x(\tau_n)\in C_{W}^{+}(\tau_n)$. In this case, we can prove that
  \begin{equation}\label{nonnegativity_proof1}
    \xi_{1}(\tau_n)>W(\tau_n)
  \end{equation}
  for $n\geq 1$. To see this, note that similarly to equation~\eqref{relxi1W}, we have $\xi_{1}(\tau_n)=\langle W \rangle_{\tau_n,\tau_{n-1}}$. We use the notation $\tau_{0}=t$ here. Therefore
  \begin{equation}
    (\tau_{n-1}-s)\xi_1(\tau_n)-\int_{s}^{\tau_{n-1}}W(\sigma)\, d\sigma=\int_{\tau_n}^{s}W(\sigma)\, d\sigma-(s-\tau_n)\xi_1(\tau_n)
  \end{equation}
  for any $s\in (\tau_n,\tau_{n-1})$. But because of definition~\eqref{taunw}, we have for any $s\in (\tau_n,\tau_{n-1})$
  \begin{equation}
    (\tau_{n-1}-s)\xi_1(\tau_n)<\int_{s}^{\tau_{n-1}}W(\sigma)\, d\sigma.
  \end{equation}
  Hence we have
  \begin{equation}
    \int_{\tau_n}^{s}W(\sigma)\, d\sigma<(s-\tau_n)\xi_1(\tau_n)
  \end{equation}
  for any $s\in (\tau_n,\tau_{n-1})$. This implies, by taking the limit $s\to \tau_n$, that $W(\tau_n)\leq \xi_1(\tau_n)$. The equality can safely be avoided because it only happens on a measure theoretically negligible subset of $I_{W}^{+}(t)$ and inequality~\eqref{nonnegativity_proof1} is proved. Now since
  \begin{equation}
    \xi_{1}'(\tau_n)=2W(\tau_n)-\xi_1(\tau_n),
  \end{equation}
  we have
  \begin{equation}
    |\xi_{1}'(\tau_n)|^2=|\xi_1(\tau_n)|^2-4W(\tau_n)[\xi_1(\tau_n)-W(\tau_n)].
  \end{equation}
  Inequality~\eqref{nonnegativity_proof1} and $W(\tau_n)\geq V_{\infty}/2>0$ implies that $|\xi_{1}'(\tau_n)|<|\xi_1(\tau_n)|$.

  From what we proved above, we have
  \begin{equation}
    |\xi_{01}|\leq \cdots <|\xi_1(\tau_{n+1})|=|\xi_{1}'(\tau_n)|<|\xi_1(\tau_n)|<\cdots <|\xi_1|
  \end{equation}
  and this is what was to be proved.
\end{proof}

We next obtain an upper bound for $r_{W}^{+}(t)$.

\begin{proposition}\label{r+upperbound}
  Let $W\in \Omega(\gamma,L,A_*,B_*,V_{\infty})$. Then
  \begin{equation}\label{r+decay}
    0\leq r_{W}^{+}(t)\leq C\left[ \frac{(\gamma+\gamma^3 A_*)^3}{(1+t)^{d+2}}+L^{-3(d-1)}\frac{B_{*}^{3}}{(1+t/L)^{d-1}} \right].
  \end{equation}
\end{proposition}

\begin{proof}
  We define two sets $A_{\leq t/2}$ and $A_{>t/2}$ as follows.
  \begin{align}
    & A_{\leq t/2}=\left\{ (x,\xi)\in I_{W}^{+}(t) \mid 0<\tilde{\tau}_1\leq t/2 \right\}, \\
    & A_{>t/2}=\left\{ (x,\xi)\in I_{W}^{+}(t) \mid t/2<\tilde{\tau}_1<t \right\}.
  \end{align}
  Let $(x,\xi)\in I_{W}^{+}(t)$. Note that we have $|\xi_0|=|\xi|$ for $(x,\xi)\notin A_{\leq t/2}\cup A_{>t/2}$. Therefore $(x,\xi)\notin A_{\leq t/2}\cup A_{>t/2}$ do not contribute to the integral~\eqref{r+} defining $r_{W}^{+}(t)$. Hence we have
  \begin{align}
    r_{W}^{+}(t)
    & =\bar{C}\int_{A_{\leq t/2}}(\xi_1-W(t))^2 \left( e^{-\beta_0 |\xi_{0}|^{2}}-e^{-\beta_0 |\xi|^2} \right) \, d\xi dS \\
    & \qquad +\bar{C}\int_{A_{>t/2}}(\xi_1-W(t))^2 \left( e^{-\beta_0 |\xi_{0}|^{2}}-e^{-\beta_0 |\xi|^2} \right) \, d\xi dS \\
    & =I+II.
  \end{align}
  We first derive an estimate for $I$. Let $(x,\xi)\in A_{\leq t/2}$. By inequality~\eqref{xiperp}, we have
  \begin{equation}
    |\xi_{\perp}|<\frac{2R}{t-\tau_1}\leq \frac{4R}{t}.
  \end{equation}
  By Lemma~\ref{tau1leq}, we have
  \begin{align}
    I
    & \leq C\left[ \frac{(\gamma+\gamma^3 A_*)^3}{(1+t)^3}+L^{-3(d-1)}\frac{B_{*}^{3}}{(1+t/L)^3} \right] \int_{|\xi_{\perp}|\leq 4R/t}e^{-\beta_0 |\xi_{\perp}|^2}\, d\xi_{\perp} \\
    & \leq C\left[ \frac{(\gamma+\gamma^3 A_*)^3}{(1+t)^{d+2}}+L^{-3(d-1)}\frac{B_{*}^{3}}{(1+t/L)^3(1+t)^{d-1}} \right]
  \end{align}
  for $d\geq 3$. If $d=2$, we have
  \begin{align}
    I
    & \leq C\left[ \frac{(\gamma+\gamma^3 A_*)^3}{(1+t)^3}+L^{-3}B_{*}^{3} \right] \int_{|\xi_{\perp}|\leq 4R/t}e^{-\beta_0 |\xi_{\perp}|^2}\, d\xi_{\perp} \\
    & \leq C\left[ \frac{(\gamma+\gamma^3 A_*)^3}{(1+t)^{4}}+L^{-3}\frac{B_{*}^{3}}{1+t} \right].
  \end{align}
  Next we derive an estimate for $II$. By Lemma~\ref{tau1geq}, we have
  \begin{equation}
    II\leq C\left[ \frac{(\gamma+\gamma^3 A_*)^3}{(1+t)^{3(d+2)}}+L^{-3(d-1)}\frac{B_{*}^{3}}{(1+t/L)^{3(d-1)}} \right].
  \end{equation}
  These estimates prove inequality~\eqref{r+decay}.
\end{proof}

\subsection{Analysis of the Characteristics for $(x,\xi)\in I_{W}^{-}(t)$}\label{sec:char-}
Let us take $W$ from $\Omega(\gamma,L,A_*,B_*,V_{\infty})$. To obtain an estimate for $r_{W}^{-}(t)$, we analyze the characteristics starting from $(x,\xi)\in I_{W}^{-}(t)$. Suppose further that $\xi_1>V_{\infty}$. Then if $0<\tau_2<\tau_1$, we have $x_1(\tau_1)=0$ and $x(\tau_2)\in C_{W}^{-}(\tau_2)$ from a simple geometrical reason. See Figure~\ref{fig:tau2w}. Hence we have
\begin{equation}
  (t-\tau_2)\xi_1+\int_{\tau_2}^{t}W(\sigma)\, d\sigma=2X(t).
\end{equation}
This can also be written as
\begin{equation}\label{relxi1Ww}
  \xi_1+\langle W \rangle_{\tau_2,t}=\frac{2X(t)}{t-\tau_2}.
\end{equation}
Furthermore, the following inequality must hold.
\begin{equation}\label{xiperp-}
  |x_{\perp}-(t-\tau_2)\xi_{\perp}|<R.
\end{equation}

\begin{figure}[htbp]
  \centering
  \begin{tikzpicture}
      \fill[pattern=north east lines] (-1,0) -- (-0.5,0) -- (-0.5,2.5) -- (-1,2.5) -- (-1,0);
      \draw (-0.5,0) -- (-0.5,2.5);
      \filldraw[fill=black!10,draw=black,dotted] (2,0.5) rectangle (2.5,2);
      \filldraw[fill=black!20,draw=black] (4,0.5) rectangle (4.5,2);
      \filldraw[thick,->,>=stealth,line width=2,black] (4.5,1.8) -- node[above] {$W(t)$} (5.2,1.8);
      \filldraw[fill=black,draw=black] (4,0.8) circle (0.05);
      \draw[->,thick,>=stealth] (4,0.8) -- (4.8,0.8) node[right] {$\xi_1$};
      \draw[semithick] (0,0.8) -- (4,0.8);
      \draw[semithick] (0,0.8) arc (270:90:0.4 and 0.5);
      \draw[semithick] (0,1.8) -- (2,1.8);
      \draw[thick,black] (2,1.8) -- (4,1.8);
      \filldraw[fill=black,draw=black] (2,1.8) circle (0.05);
      \draw[->,thick,>=stealth] (2,1.8) -- (1.2,1.8);
      \node[black] at (4.25,0.2) {$s=t$};
      \node[right,black] at (-0.4,1.1) {$s=\tau_1$};
      \node[above,black] at (2,1.8) {$s=\tau_2$};
      \draw[<->,thick] (-0.5,2.6) -- node[above] {$X(t)$} (4,2.6);
      \draw[thick,dotted] (-0.5,2.5) -- (-0.5,2.9);
      \draw[thick,dotted] (4,2) -- (4,2.9);
    \end{tikzpicture}
    \caption{Two dimensional picture of a pre-collision at $C_{W}^{-}(\tau_2)$ via pre-collision at the plane wall. The sum of the horizontal distance traversed by the cylinder and the characteristic curve $x(s)$ from $\tau_2$ to $t$ equals $2X(t)$.}
  \label{fig:tau2w}
\end{figure}
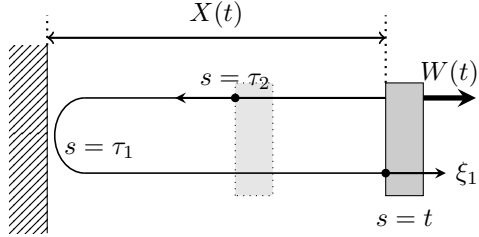

Now we prove the following lemma.

\begin{lemma}\label{tau2wnecessary}
  Let $W\in \Omega(\gamma,L,A_*,B_*,V_{\infty})$, $(x,\xi)\in I_{W}^{-}(t)$ and $\xi_1>V_{\infty}$. If $\tau_2>0$, then we have
  \begin{equation}\label{tau2wgeq}
    t-\tau_2\geq \frac{2(L+V_{\infty}t/2)}{\xi_1+V_{\infty}}.
  \end{equation}
\end{lemma}

\begin{proof}
  By equation~\eqref{relxi1Ww}, we have
  \begin{equation}
    t-\tau_2=\frac{2X(t)}{\xi_1+\langle W \rangle_{\tau_2,t}}.
  \end{equation}
  Since $W(s)\geq V_{\infty}/2$, we have
  \begin{equation}
    X(t)=L+\int_{0}^{t}W(s)\, ds\geq L+V_{\infty}t/2.
  \end{equation}
  Moreover, since $W(s)<V_{\infty}$, we have $\langle W \rangle_{\tau_2,t}<V_{\infty}$. From these we obtain inequality~\eqref{tau2wgeq}.
\end{proof}

The lemma above shows a necessary condition for $\tau_2>0$. We give below a sufficient condition. This will be needed when we derive a lower bound for $r_{W}^{-}(t)$ in Section~\ref{sec:r-lowerbound}.

\begin{lemma}\label{tau2wsufficient}
  Let $W\in \Omega(\gamma,L,A_*,B_*,V_{\infty})$, $(x,\xi)\in I_{W}^{-}(t)$ and $\xi_1>V_{\infty}$. If the following conditions are satisfied, then $\tau_2>0$.
  \begin{enumerate}[label={\upshape(\roman*)}]
    \item $|x_{\perp}|\leq R/2$,
    \item $\displaystyle |\xi_{\perp}|\leq \frac{R(\xi_1+V_{\infty}/2)}{4(L+V_{\infty}t)}$ and
    \item $\displaystyle \xi_1\geq \frac{3}{2}V_{\infty}+\frac{2L}{t}\left[ =\frac{2(L+V_{\infty}t)}{t}-\frac{V_{\infty}}{2} \right]$.
  \end{enumerate}
\end{lemma}

\begin{proof}
  Condition (iii) implies that
  \begin{equation}
    \xi_1+\langle W \rangle_{0,t}\geq \xi_1+\frac{V_{\infty}}{2}\geq \frac{2(L+V_{\infty}t)}{t}>\frac{2X(t)}{t}.
  \end{equation}
  By the intermediate value theorem, there exists $s\in (0,t)$ such that
  \begin{equation}\label{pseudotau2weq}
    \xi_1+\langle W \rangle_{s,t}=\frac{2X(t)}{t-s}.
  \end{equation}
  Let $\sigma_2$ be the largest $s\in (0,t)$ satisfying equation~\eqref{pseudotau2weq}. If we can show that
  \begin{equation}\label{pseudotau2wperp}
    |x_{\perp}-(t-\sigma_2)\xi_{\perp}|<R,
  \end{equation}
  we can easily see that $\tau_2=\sigma_2\in (0,t)$ and the lemma is proved. To prove inequality~\eqref{pseudotau2wperp}, note that by conditions (i) and (ii), we have
  \begin{equation}
    |x_{\perp}-(t-\sigma_2)\xi_{\perp}|\leq \frac{R}{2}+\frac{R(\xi_1+V_{\infty}/2)}{4(L+V_{\infty}t)}(t-\sigma_2)
  \end{equation}
  and by the definition of $\sigma_2$, we have
  \begin{equation}
    t-\sigma_2=\frac{2X(t)}{\xi_1+\langle W \rangle_{\sigma_2,t}}<\frac{2(L+V_{\infty}t)}{\xi_1+V_{\infty}/2}.
  \end{equation}
  These imply inequality~\eqref{pseudotau2wperp}.
\end{proof}

\subsection{Non-Negativity of $r_{W}^{-}(t)$ and its Upper Bound}\label{sec:r-}
First we prove the non-negativity of $r_{W}^{-}(t)$.

\begin{proposition}\label{r-nonnegative}
  For $W\in \Omega(\gamma,L,A_*,B_*,V_{\infty})$, we have $r_{W}^{-}(t)\geq 0$.
\end{proposition}

\begin{proof}
  The proof is basically the same as the proof of Proposition~\ref{r+nonnegative}. It suffices to note that if $x_1(\tau_n)=0$ and $x(\tau_{n+1})\in C_{W}^{-}(\tau_{n+1})$, then
  \begin{equation}
    \xi_{1}'(\tau_{n+1})=2W(\tau_{n+1})+\xi_1(\tau_n)
  \end{equation}
  and therefore $|\xi_{1}'(\tau_{n+1})|>|\xi_{1}(\tau_n)|$.
\end{proof}

We next prove an upper bound for $r_{W}^{-}(t)$.

\begin{proposition}\label{r-upperbound}
  Let $W\in \Omega(\gamma,L,A_*,B_*,V_{\infty})$. Then
  \begin{equation}
    0\leq r_{W}^{-}(t)\leq C\left[ \frac{(\gamma+\gamma^3 A_*)^3}{(1+t)^{3(d+2)}}+L^{-3(d-1)}\frac{B_{*}^{3}}{(1+t/L)^{3(d-1)}}+\frac{L^{-(d-1)}}{(1+t/L)^{d-1}} \right].
  \end{equation}
\end{proposition}

\begin{proof}
  We split $r_{W}^{-}(t)$ into two parts.
  \begin{align}
    r_{W}^{-}(t)
    & =\bar{C}\int_{C_{W}^{-}(t)}\, dS\left[ \int_{W(t)<\xi_1<V_{\infty}}(\xi_1-W(t))^2 \left( e^{-\beta_0 |\xi|^2}-e^{-\beta_0 |\xi_{0}|^2} \right) \, d\xi \right. \\
    & \quad \qquad \qquad \qquad \qquad +\left. \int_{\xi_1>V_{\infty}}(\xi_1-W(t))^2 \left( e^{-\beta_0 |\xi|^2}-e^{-\beta_0 |\xi_{0}|^2} \right) \, d\xi \right] \\
    & =I+II.
  \end{align}
  By inequality~\eqref{upperbound}, it follows that
  \begin{equation}
    I\leq C(V_{\infty}-W(t))^3 \leq C\left[ \frac{(\gamma+\gamma^3 A_*)^3}{(1+t)^{3(d+2)}}+L^{-3(d-1)}\frac{B_{*}^{3}}{(1+t/L)^{3(d-1)}} \right].
  \end{equation}
  For $II$, we use Lemma~\ref{tau2wnecessary}. Note that since $\xi_1>V_{\infty}$, we have $\xi_{01}=\pm \xi_1$ if $\tau_1=0$ or $\tau_2=0$ and there is no contribution to the integral defining $II$. Therefore we can assume that $\tau_2>0$. Thus equation~\eqref{xiperp-} and Lemma~\ref{tau2wnecessary} imply that
  \begin{equation}
    |\xi_{\perp}|\leq \frac{2R}{t-\tau_2}\leq \frac{2R(\xi_1+V_{\infty})}{2L+V_{\infty}t}
  \end{equation}
  and we have
  \begin{align}
    II
    & \leq C\int_{V_{\infty}}^{\infty}(\xi_1-W(t))^2 e^{-\beta_0 \xi_{1}^{2}}\, d\xi_1 \int_{|\xi_{\perp}|\leq \frac{2R(\xi_1+V_{\infty})}{2L+V_{\infty}t}}e^{-\beta_0 |\xi_{\perp}|^{2}}\, d\xi_{\perp} \\
    & \leq C\frac{L^{-(d-1)}}{(1+t/L)^{d-1}}\int_{V_{\infty}}^{\infty}(\xi_1-V_{\infty}/2)^2 (\xi_1+V_{\infty})^{d-1}e^{-\beta_0 \xi_{1}^{2}}\, d\xi_1 \\
    & \leq C\frac{L^{-(d-1)}}{(1+t/L)^{d-1}}.
  \end{align}
  These estimates for $I$ and $II$ prove the proposition.
\end{proof}

\subsection{Proof of $V_W \in \Omega$}\label{sec:VW}
Let $W\in \Omega=\Omega(\gamma,L,A_*,B_*,V_{\infty})$. We prove here that $V_W\in \Omega$ for sufficiently small $\gamma$ and sufficiently large $L$. Two positive constants $A_+=A_+(\beta_0,p_0,R,V_{\infty},d)$ and $B_+=B_+(\beta_0,p_0,R,V_{\infty},d)$ are chosen appropriately and set equal to $A_*$ and $B_*$ respectively.

\begin{proposition}\label{VWOmega}
  There exist positive constants $\gamma_0$, $L_0$, $A_+$ and $B_+$ depending only on $\beta_0$, $p_0$, $R$, $V_{\infty}$ and $d$ such that for all $\gamma \in (0,\gamma_0]$ and $L\in [L_0,\infty)$, $W\in \Omega=\Omega(\gamma,L,A_+,B_+,V_{\infty})$ implies $V_W \in \Omega$, where $V_W(t)$ is defined by equations~\eqref{map}, \eqref{mapIC} and \eqref{KW}. Moreover, inequality~\eqref{upperbound} is strict, that is,
  \begin{equation}
    V_{\infty}-V_W(t)<\gamma e^{-C_{+}t}+\gamma^3 \frac{A_+}{(1+t)^{d+2}}+L^{-(d-1)}\frac{B_+}{(1+t/L)^{d-1}}
  \end{equation}
  for $t\geq 0$.
\end{proposition}

\begin{proof}
  Let $W\in \Omega(\gamma,L,A_*,B_*,V_{\infty})$. First, note that by the non-negativity of $r_{W}^{\pm}(t)$ (Propositions~\ref{r+nonnegative} and~\ref{r-nonnegative}) and Lemma~\ref{D0properties}, we have
  \begin{align}
    V_{\infty}-V_W(t)
    & =\gamma e^{-\int_{0}^{t}K_W(\sigma)\, d\sigma}+\int_{0}^{t}e^{-\int_{s}^{t}K_W(\sigma)\, d\sigma}\left[ r_{W}^{+}(s)+r_{W}^{-}(s) \right] \, ds \\
    & \geq \gamma e^{-C_{-}t}.
  \end{align}
  Therefore, inequality~\eqref{lowerbound1} is proved for $V_W(t)$. We next prove inequality~\eqref{upperbound}. By Propositions~\ref{r+upperbound} and \ref{r-upperbound}, we have for both cases $d\geq 3$ and $d=2$
  \begin{align}
    V_{\infty}-V_W(t)
    & =\gamma e^{-\int_{0}^{t}K_W(\sigma)\, d\sigma}+\int_{0}^{t}e^{-\int_{s}^{t}K_W(\sigma)\, d\sigma}\left[ r_{W}^{+}(s)+r_{W}^{-}(s) \right] \, ds \\
    & \leq \gamma e^{-C_{+}t}+C(\gamma+\gamma^3 A_*)^3 \int_{0}^{t}\frac{e^{-C_{+}(t-s)}}{(1+s)^{d+2}}\, ds \\
    & \qquad +CL^{-3(d-1)}B_{*}^{3}\int_{0}^{t}\frac{e^{-C_{+}(t-s)}}{(1+s/L)^{d-1}}\, ds \\
    & \qquad +CL^{-(d-1)}\int_{0}^{t}\frac{e^{-C_{+}(t-s)}}{(1+s/L)^{d-1}}\, ds.
  \end{align}
  Splitting the integral at $t/2$, we obtain
  \begin{align}
    V_{\infty}-V_W(t)\leq \gamma e^{-C_{+}t}+\tilde{C}
    & \left[ (\gamma+\gamma^3 A_*)^3 \frac{1}{(1+t)^{d+2}} \right. \\
    & \qquad +\left. L^{-3(d-1)}\frac{B_{*}^{3}}{(1+t/L)^{d-1}}+\frac{L^{-(d-1)}}{(1+t/L)^{d-1}} \right]
  \end{align}
  for some $\tilde{C}=\tilde{C}(\beta_0,p_0,R,V_{\infty},d)>0$. Now we take $A_+=B_+=A_*=B_*=2\tilde{C}$. Note that $A_+$ and $B_+$ depends only on $\beta_0$, $p_0$, $R$, $V_{\infty}$ and $d$. Finally, we take $\gamma_0$ small enough and $L_0$ large enough so that
  \begin{align}
    \tilde{C}(1+\gamma_{0}^{2}A_+)^3 & <A_+, \\
    \tilde{C}\left( L_{0}^{-2(d-1)}B_{+}^{3}+1 \right) & <B_+.
  \end{align}
  Note that this smallness and largeness depend only on the above mentioned parameters. With this $\gamma_0$ and $L_0$, we have for all $\gamma \in (0,\gamma_0]$ and $L\in [L_0,\infty)$
  \begin{equation}
    V_{\infty}-V_W(t)<\gamma e^{-C_{+}t}+\gamma^3 \frac{A_+}{(1+t)^{d+2}}+L^{-(d-1)}\frac{B_+}{(1+t/L)^{d-1}}.
  \end{equation}
  This proves inequality~\eqref{upperbound} for $V_W(t)$ and the inequality holds strictly. Note that $V_W(t)$ is Lipschitz continuous because $r_{W}^{\pm}(t)$ are bounded. It remains to show that $V_W(t)\geq V_{\infty}/2$. Take $\gamma_0$ small enough and $L_0$ large enough so that
  \begin{equation}
    \gamma_0+\gamma_{0}^{3}A_{+}+L_{0}^{-(d-1)}B_{+}\leq V_{\infty}/2.
  \end{equation}
  Then we have by inequality~\eqref{upperbound} that $V_{\infty}-V_W(t)\leq V_{\infty}/2$ for all $\gamma \in (0,\gamma_0]$ and $L\in [L_0,\infty)$. Thus we have $V_W(t)=V_{\infty}-(V_{\infty}-V_W(t))\geq V_{\infty}/2$.
\end{proof}

\subsection{Proof of the First Part of Theorem~\ref{theorem1}}\label{sec:proof1}
We prove here the existence of a fixed point for the map $W\mapsto V_W$. Let $C_b([0,\infty))$ be the space of bounded continuous functions on the interval $[0,\infty)$. Let $\gamma_0$, $L_0$, $A_+$ and $B_+$ be as in Proposition~\ref{VWOmega}. For $\gamma \in (0,\gamma_0]$ and $L\in [L_0,\infty)$, define a convex subset $\mathscr{K}$ of $C_b([0,\infty))$ by
\begin{equation}
  \mathscr{K}=\left\{ W\in \Omega(\gamma,L,A_+,B_+,V_{\infty}) \mathrel{\Big|} \operatorname{ess\ sup}_{0\leq t<\infty}\left( |W(t)|+|dW(t)/dt| \right) \leq M \right\}
\end{equation}
and
\begin{equation}
  M=V_{\infty}+C_{-}\gamma +C\left[ (\gamma+\gamma^3 A_+)^3 +L^{-3(d-1)}B_{+}^{3}+L^{-(d-1)} \right].
\end{equation}
By Propositions~\ref{r+upperbound}, \ref{r-upperbound} and~\ref{VWOmega}, the map $W\mapsto V_W$ maps $\mathscr{K}$ into itself. The Arzel\`{a}--Ascoli theorem implies that $\mathscr{K}$ is a compact subset of $C_b([0,\infty))$.\footnote{We can apply the Arzel\`{a}--Ascoli theorem on a non-compact space $[0,\infty)$ because $V_{\infty}-W(t)$ is uniformly decaying for $W\in \mathscr{K}$.}

We prove next that the map $\mathscr{K}\ni W\mapsto V_W \in \mathscr{K}$ is continuous in the topology of $C_b([0,\infty))$.

\begin{proposition}\label{mapcont}
  Let $\{ W_j \}_{j=1}^{\infty}\subset \mathscr{K}$, $W\in \mathscr{K}$ and $W_j \to W$ in $C_b([0,\infty))$. Then we have $r_{W_j}^{\pm}(t)\to r_{W}^{\pm}(t)$ for all $t\geq 0$.
\end{proposition}

\begin{proof}
  We only treat $r_{W}^{+}(t)$ here because $r_{W}^{-}(t)$ can be handled similarly. Let $(x,\xi)\in I_{W}^{+}(t)$. Note that for sufficiently large $j$, we have $\xi_1<W_j(t)$. Take $m$ to be the integer satisfying $\tau_m>0$ and $\tau_{m+1}=0$ for the dynamics given by $W(t)$. Let $\tau_{n,j}$ denote the $n$-th pre-collision time for the dynamics given by $W_j(t)$. We now prove that for sufficiently large $j$, we have $\tau_{m,j}>0$ and $\tau_{m+1,j}=0$. Moreover, we show that $\tau_{n,j}\to \tau_n$ for $1\leq n\leq m$.
  
  We first treat the case of $m=0$, the case without pre-collisions. We first note that $m=0$ implies
  \begin{equation}
    \xi_1<\frac{X(t)+h}{t}
  \end{equation}
  because otherwise there would be a pre-collision at the plane wall. Note that we also have
  \begin{equation}
    \xi_1<\frac{X_j(t)+h}{t}
  \end{equation}
  for sufficiently large $j$, where $X_j(t)=L+\int_{0}^{t}W_j(s)\, ds$. In order to have $m=0$, we have two possibilities, that is, the characteristic curve $x(s)$ never catches up the cylinder in the $x_1$-direction (i); or it does catch up the cylinder in the $x_1$-direction but escapes in the $x_{\perp}$-direction (ii). Expressed in equations, these are
  (i):
  \begin{equation}
    \xi_1<\inf_{0\leq s<t}\langle W \rangle_{s,t} 
  \end{equation}
  or; (ii):
  \begin{equation}\label{x1precollision}
    \inf_{0\leq s<t}\langle W \rangle_{s,t}<\xi_1
  \end{equation}
  but
  \begin{equation}\label{noprecollisionperp}
    |x_{\perp}-(t-\sigma_1)\xi_{\perp}|>R.
  \end{equation}
  Here $\sigma_1$ is the largest $s\in (0,t)$ satisfying $\xi_1=\langle W \rangle_{s,t}$, which exists by inequality~\eqref{x1precollision}. For the case (i), we have
  \begin{equation}
    \xi_1<\inf_{0\leq s<t}\langle W_j \rangle_{s,t}
  \end{equation}
  for sufficiently large $j$. Hence there also are no pre-collisions for the dynamics given by $W_j(t)$. For the case (ii), we have
  \begin{equation}
    \inf_{0\leq s<t}\langle W_j \rangle_{s,t}<\xi_1
  \end{equation}
  for sufficiently large $j$. Let $\sigma_{1,j}$ be the largest $s\in (0,t)$ satisfying $\xi_1=\langle W_j \rangle_{s,t}$. By the definition of $\sigma_{1,j}$, we have
  \begin{align}
    \xi_1 & <\langle W_j \rangle_{s,t} \quad \text{for $s\in (\sigma_{1,j},t)$}, \\
    \xi_1 & =\langle W_j \rangle_{\sigma_{1,j},t}.
  \end{align}
  Hence for any convergent subsequence $\{ \sigma_{1,j'} \}$ of $\{ \sigma_{1,j} \}$, we have
  \begin{align}
    \xi_1 & <\langle W \rangle_{s,t} \quad \text{for $s\in (\sigma_{1*},t)$}, \\
    \xi_1 & =\langle W \rangle_{\sigma_{1*},t},
  \end{align}
  where $\sigma_{1*}$ is the limit of the subsequence. We excluded the equality in the first inequality because it only happens on a measure theoretically negligible set.\footnote{If $\xi_1=\langle W \rangle_{\tilde{s},t}$ for some $\tilde{s}\in (\sigma_{1*},t)$, it follows that $d\langle W \rangle_{s,t}/ds=0$ at $s=\tilde{s}$. This implies that $\xi_1=\langle W \rangle_{\tilde{s},t}=W(\tilde{s})$.} This shows that $\sigma_{1*}=\sigma_1$ and hence $\sigma_{1,j}\to \sigma_1$. By inequality~\eqref{noprecollisionperp}, we have
  \begin{equation}
    |x_{\perp}-(t-\sigma_{1,j})\xi_{\perp}|>R
  \end{equation}
  for sufficiently large $j$. Thus there also are no pre-collisions for the dynamics given by $W_j(t)$.

  We consider next the case of $m=1$. We first treat the case of $x_1(\tau_1)=0$. In this case, we have
  \begin{equation}
    \xi_1=\frac{X(t)+h}{t-\tau_1}>\langle W \rangle_{\tau_1,t}\geq \inf_{0\leq s<t}\langle W \rangle_{s,t}.
  \end{equation}
  Let $\eta_1$ be the largest $s\in (0,t)$ satisfying $\xi_1=\langle W \rangle_{s,t}$. Note that we have
  \begin{equation}\label{noprecollisionperp2}
    |x_{\perp}-(t-\eta_1)\xi_{\perp}|>R.
  \end{equation}
  Define $\sigma_{1,j}$ by the equation
  \begin{equation}
    t-\sigma_{1,j}=\frac{X_j(t)+h}{\xi_1}.
  \end{equation}
  Assuming that $\tau_1>0$, we have $\sigma_{1,j}>0$ for sufficiently large $j$. Note that $\tau_1=0$ implies $\xi_1=(X(t)+h)/t$ and therefore this case is measure theoretically negligible. Note also that
  \begin{equation}
    \xi_1>\inf_{0\leq s<t}\langle W_j \rangle_{s,t}
  \end{equation}
  for sufficiently large $j$. Let $\eta_{1,j}$ be the largest $s\in (0,t)$ satisfying $\xi_1=\langle W_j \rangle_{s,t}$. A similar argument as in the case of $m=0$ shows that $\eta_{1,j}\to \eta_1$. Hence inequality~\eqref{noprecollisionperp2} implies
  \begin{equation}
    |x_{\perp}-(t-\eta_{1,j})\xi_{\perp}|>R
  \end{equation}
  for sufficiently large $j$. These show that $\tau_{1,j}=\sigma_{1,j}>0$ and that $x_1(\tau_{1,j})=0$. For a simple geometrical reason, there are no pre-collisions after a pre-collision at the plane wall. Hence $\tau_{2,j}=0$. We next treat the case of $x(\tau_{1})\in C_{W}^{+}(\tau_{1})$. In this case, we have
  \begin{equation}
    \inf_{0\leq s<t}\langle W \rangle_{s,t}<\xi_1
  \end{equation}
  and
  \begin{equation}\label{precollisionperp}
    |x_{\perp}-(t-\tau_1)\xi_{\perp}|<R.
  \end{equation}
  Therefore we have
  \begin{equation}
    \inf_{0\leq s<t}\langle W_j \rangle_{s,t}<\xi_1
  \end{equation}
  for sufficiently large $j$. Let $\sigma_{1,j}$ be the largest $s\in (0,t)$ satisfying $\xi_1=\langle W_j \rangle_{s,t}$. A similar argument as in the case of $m=0$ shows that $\sigma_{1,j}\to \tau_{1}$. By inequality~\eqref{precollisionperp}, we have
  \begin{equation}
    |x_{\perp}-(t-\sigma_{1,j})\xi_{\perp}|<R
  \end{equation}
  for sufficiently large $j$. This shows that $\tau_{1,j}=\sigma_{1,j}>0$. Lastly, an argument similar to the case of $m=0$ shows that $\tau_{2,j}=0$ for sufficiently large $j$. Note also that since $\tau_{1,j}\to \tau_{1}$, we have
  \begin{equation}
    2W_j(\tau_{1,j})-\xi_1 \to 2W(\tau_{1})-\xi_1=\xi_{1}'(\tau_{1}).
  \end{equation}
  Therefore we can extend the argument employed here to treat general $m$.

  Let $\xi_{0,j}=\xi_{j}'(\tau_{m,j})$. Here the characteristics $(x_j(s),\xi_j(s))$ are defined using $W_j$. From what we showed above, we have
  \begin{equation}
    \xi_{0,j}\to \xi_0 \quad \text{a.e. $(x,\xi)\in I_{W}^{+}(t)$}.
  \end{equation}
  Hence by the Lebesgue convergence theorem, we conclude that
  \begin{equation}
    r_{W_j}^{+}(t)\to r_{W}^{+}(t)
  \end{equation}
  for all $t\geq 0$.
\end{proof}

\begin{remark}\label{remark:cont}
  A similar argument as in the proof above shows that $r_{W}^{\pm}(t)$ are continuous in $t$. Therefore $V(t)$ is necessarily continuously differentiable if $(X(t),V(t))$ is a solution to the problem.
\end{remark}

\begin{proposition}
  Let $\{ W_j \}_{j=1}^{\infty}\subset \mathscr{K}$, $W\in \mathscr{K}$ and $W_j \to W$ in $C_b([0,\infty))$. Then we have $V_{W_j}\to V_W$ in $C_b([0,\infty))$.
\end{proposition}

\begin{proof}
  Duhamel's formula implies that
\begin{align}
  V_W(t)-V_{W_j}(t)
  & =\gamma \left[ e^{-\int_{0}^{t}K_{W_j}(\sigma)\, d\sigma}-e^{-\int_{0}^{t}K_{W}(\sigma)\, d\sigma} \right] \\
  & \qquad +\int_{0}^{t}e^{-\int_{s}^{t}K_{W_j}(\sigma)\, d\sigma}\left[ r_{W_j}^{+}(s)+r_{W_j}^{-}(s) \right] \, ds \\
  & \qquad -\int_{0}^{t}e^{-\int_{s}^{t}K_{W}(\sigma)\, d\sigma}\left[ r_{W}^{+}(s)+r_{W}^{-}(s) \right] \, ds.
\end{align}
We show here that
\begin{equation}\label{mapcont_estimate}
  \int_{0}^{t}e^{-\int_{s}^{t}K_W(\sigma)\, d\sigma}\left| r_{W_j}^{+}(s)-r_{W}^{+}(s) \right| \, ds \to 0
\end{equation}
as $j\to \infty$ uniformly in $t\geq 0$. All other terms are similarly treated or are easier to treat. Now for any $\varepsilon>0$, take $T>0$ such that for all $t\geq T$
\begin{equation}
  \int_{T}^{t}e^{-C_{+}(t-s)}\left[ r_{W_j}^{+}(s)+r_{W}^{+}(s) \right] \, ds<\varepsilon/2.
\end{equation}
This is possible because by Proposition~\ref{r+upperbound}, $r_{W_j}^{+}(t)$ and $r_{W}^{+}(t)$ decay as $t\to \infty$ uniformly in $j$. Now by the Lebesgue convergence theorem and Proposition~\ref{mapcont}, there exists $N=N(T)\in \mathbb{N}$ such that for all $j\geq N$
\begin{equation}
  \int_{0}^{T}\left| r_{W_j}^{+}(s)-r_{W}^{+}(s) \right| \, ds<\varepsilon/2.
\end{equation}
Hence if $j\geq N$, we have
\begin{equation}
  \int_{0}^{t}e^{-\int_{s}^{t}K_W(\sigma)\, d\sigma}\left| r_{W_j}^{+}(s)-r_{W}^{+}(s) \right| \, ds<\varepsilon/2
\end{equation}
for $t\leq T$. This proves~\eqref{mapcont_estimate} and $V_{W_j}\to V_W$ in $C_b([0,\infty))$ follows.
\end{proof}

Now applying Schauder's fixed point theorem shows that there exists a fixed point $V\in \mathscr{K}$ for the map $W\mapsto V_W$. This proves the first part of Theorem~\ref{theorem1}.

\subsection{Proof of the Second Part of Theorem~\ref{theorem1}}
Let $(X(t),V(t))$ be any solution to the problem. Define $T$ by
\begin{equation}
  T=\inf \left\{ t\geq 0 \mathrel{\Bigg|} V_{\infty}-V(t)\geq \gamma e^{-C_{+}t}+\gamma^3 \frac{A_+}{(1+t)^{d+2}}+L^{-(d-1)}\frac{B_+}{(1+t/L)^{d-1}} \right\}.
\end{equation}
Here we use the convention that the infimum of the empty set equals $+\infty$. It is obvious that $T>0$. Suppose that $T<+\infty$. By the definition of $T$, we have
\begin{equation}\label{ineqT}
  V_{\infty}-V(t)\leq \gamma e^{-C_{+}t}+\gamma^3 \frac{A_+}{(1+t)^{d+2}}+L^{-(d-1)}\frac{B_+}{(1+t/L)^{d-1}}
\end{equation}
for $t\leq T$ and
\begin{equation}
  V_{\infty}-V(T)=\gamma e^{-C_{+}T}+\gamma^3 \frac{A_+}{(1+T)^{d+2}}+L^{-(d-1)}\frac{B_+}{(1+T/L)^{d-1}}.
\end{equation}
By inequality~\eqref{ineqT}, taking $\gamma_0$ small enough and $L_0$ large enough, we have $V(t)\geq V_{\infty}/2$ for all $\gamma \in (0,\gamma_0]$, $L\in [L_0,\infty)$ and $t\leq T$. Using this, we can prove as in the proof of Propositions~\ref{r+nonnegative} and~\ref{r-nonnegative} that $r_{V}^{\pm}(t)\geq 0$ for $t\leq T$. Hence from the equation
\begin{equation}
  V_{\infty}-V(t)=\gamma e^{-\int_{0}^{t}K_V(\sigma)\, d\sigma}+\int_{0}^{t}e^{-\int_{s}^{t}K_V(\sigma)\, d\sigma}\left[ r_{V}^{+}(s)+r_{V}^{-}(s) \right] \, ds,
\end{equation}
it follows that $V(t)<V_{\infty}$ for $t\leq T$. Since $V_{\infty}/2\leq V(t)<V_{\infty}$, we have $C_{+}\leq K_V(\sigma)\leq C_{-}$. Hence inequality~\eqref{estimate2} holds for $t\leq T$. Now Proposition~\ref{VWOmega}\footnote{More precisely, we use a version of Proposition~\ref{VWOmega} where all the inequalities are modified to hold up to $t=T$.} shows that
\begin{equation}
  V_{\infty}-V(T)<\gamma e^{-C_{+}T}+\gamma^3 \frac{A_+}{(1+T)^{d+2}}+L^{-(d-1)}\frac{B_+}{(1+T/L)^{d-1}},
\end{equation}
which is a contradiction. Therefore $T=+\infty$. Hence inequality~\eqref{ineqT} holds for $t\geq 0$. Taking $\gamma_0$ small enough and $L_0$ large enough if necessary, it follows that $V(t)\geq V_{\infty}/2$ for $t\geq 0$. This shows that $r_{V}^{\pm}(t)\geq 0$ and $V(t)<V_{\infty}$ for $t\geq 0$. As a consequence, inequality~\eqref{estimate2} holds for $t\geq 0$. This concludes the proof of Theorem~\ref{theorem1}.

\subsection{Lower Bound for $r_{W}^{-}(t)$}\label{sec:r-lowerbound}
We prove here a lower bound for $r_{W}^{-}(t)$.

\begin{proposition}\label{r-lowerbound}
  Let $W\in \Omega(\gamma,L,A_*,B_*,V_{\infty})$. Then we have
  \begin{equation}
    r_{W}^{-}(t)\geq C\bm{1}_{\{ t>L/2 \}}\frac{L^{-(d-1)}}{(1+t/L)^{d-1}}.
  \end{equation}
\end{proposition}

\begin{proof}
  Let $(x,\xi)\in I_{W}^{-}(t)$ and suppose that conditions (i), (ii) and (iii) in Lemma~\ref{tau2wsufficient} are satisfied. Then $\tau_2>0$ by Lemma~\ref{tau2wsufficient}. Since
  \begin{equation}
    |\xi_{01}|\geq |\xi_{1}'(\tau_2)|=|2W(\tau_2)+\xi_1|\geq V_{\infty}+\xi_1,
  \end{equation}
   we have
  \begin{align}
    e^{-\beta_0 \xi_{1}^{2}}-e^{-\beta_0 \xi_{01}^{2}}
    & \geq e^{-\beta_0 \xi_{1}^{2}}-e^{-\beta_0(V_{\infty}+\xi_1)^2} \\
    & \geq \beta_0 e^{-\beta_0 (V_{\infty}+\xi_1)^2}\left[ (V_{\infty}+\xi_1)^2 -\xi_{1}^{2} \right] \\
    & =\beta_0 V_{\infty}(V_{\infty}+2\xi_1)e^{-\beta_0 (V_{\infty}+\xi_1)^2}.
  \end{align}
  Hence for $t>L/2$, we have
  \begin{align}
    r_{W}^{-}(t)
    & \geq C\int_{\frac{3}{2}V_{\infty}+\frac{2L}{t}}^{\infty}(V_{\infty}+2\xi_1)(\xi_1-W(t))^2 e^{-\beta_0 (V_{\infty}+\xi_1)^2}\, d\xi_1 \\
    & \qquad \qquad \times \int_{|\xi_{\perp}|\leq \frac{R(\xi_1+V_{\infty}/2)}{4(L+V_{\infty}t)}}e^{-\beta_0 |\xi_{\perp}|^{2}}\, d\xi_{\perp} \\
    & \geq C\frac{L^{-(d-1)}}{(1+t/L)^{d-1}}\int_{\frac{3}{2}V_{\infty}+4}^{\infty}(V_{\infty}+2\xi_1)(\xi_1-V_{\infty}/2)^2(\xi_1+V_{\infty}/2)^{d-1} \\
    & \quad \qquad \qquad \qquad \qquad \qquad \times e^{-\beta_0 (V_{\infty}+\xi_1)^2}e^{-\beta_0 R^2 (\xi_1+V_{\infty}/2)^2}\, d\xi_1 \\
    & \geq C\frac{L^{-(d-1)}}{(1+t/L)^{d-1}}.
  \end{align}
\end{proof}

\subsection{Proof of Theorem~\ref{theorem2}}\label{sec:proof2}
Let $(X(t),V(t))$ be any solution to the problem. By Theorem~\ref{theorem1}, we see that $V\in \Omega(\gamma,L,A_+,B_+,V_{\infty})$. Hence by Proposition~\ref{r-lowerbound}, we have for $t>L$
\begin{align}
  V_{\infty}-V(t)
  & \geq \gamma e^{-C_{-}t}+\int_{0}^{t}e^{-C_{-}(t-s)}r_{V}^{-}(s)\, ds \\
  & \geq \gamma e^{-C_{-}t}+CL^{-(d-1)}\int_{L/2}^{t}\frac{e^{-C_{-}(t-s)}}{(1+s/L)^{d-1}}\, ds \\
  & \geq \gamma e^{-C_{-}t}+C\frac{L^{-(d-1)}}{(1+t/L)^{d-1}}\frac{1-e^{-C_{-}L/2}}{C_-} \\
  & \geq \gamma e^{-C_{-}t}+\frac{B_-}{t^{d-1}},
\end{align}
where $B_{-}=B_{-}(\beta_0,p_0,R,V_{\infty},d)$ is a positive constant.
This concludes the proof of Theorem~\ref{theorem2}.

\section{Appendix}\label{appendix}
We prove here several assertions made in Remarks~\ref{remark2} and \ref{remark3}. In this appendix, we consider $L=L_{\gamma}$ to be a function of $\gamma$ and write $L^{-(d-1)}\ll \gamma$ to mean $L_{\gamma}^{-(d-1)}=o(\gamma)$ as $\gamma \to +0$. In the following, $(X(t),V(t))$ is a solution to the problem, and $K_V(t)$ is defined by equation~\eqref{KW} using $V$ in place of $W$.

\subsection{Proof of $V(t)>V_0$}\label{appendix1}
We first prove that if $L^{-(d-1)}\ll \gamma$, then $V(t)>V_0$ for $t>0$. First, by equations~\eqref{EOM}, \eqref{terminal} and $D_V(t)=D_0(V(t))+r_{V}^{+}(t)+r_{V}^{-}(t)$, we have
\begin{align}\label{eq:appendix1}
  \begin{aligned}
    \frac{d}{dt}[V(t)-V_0]
    & =K_V(t)[V_{\infty}-V(t)]-r_{V}^{+}(t)-r_{V}^{-}(t) \\
    & =-K_V(t)[V(t)-V_0]+\gamma K_V(t)-r_{V}^{+}(t)-r_{V}^{-}(t).
  \end{aligned}
\end{align}
Now $V\in \Omega(\gamma,L,A_+,B_+,V_{\infty})$ by Theorem~\ref{theorem1}, and we can use Propositions~\ref{r+upperbound} and \ref{r-upperbound} to obtain upper bounds for $r_{V}^{+}(t)$ and $r_{V}^{-}(t)$. These and $K_V(t)\geq C_+$ show that
\begin{equation}
  \gamma K_V(t)-r_{V}^{+}(t)-r_{V}^{-}(t) \geq \gamma C_+ -C\left[ (\gamma+\gamma^3 A_+)^3 +L^{-3(d-1)}B_{+}^{3}+L^{-(d-1)} \right] >0
\end{equation}
for $\gamma$ sufficiently small. Note that we used $L^{-(d-1)}\ll \gamma$ here. Equation~\eqref{eq:appendix1} now implies that
\begin{equation}
  \frac{d}{dt}[V(t)-V_0]>-K_V(t)[V(t)-V_0],
\end{equation}
and $V(t)>V_0$ for $t>0$ follows from this.

\subsection{Proof that $V(t)$ is increasing on $[0,t_0]$}\label{appendix2}
As we wrote in Remark~\ref{remark2}, $V(t)$ is increasing on a time interval $[0,t_0]$, where
\begin{equation}\label{t0}
  t_0=\frac{1}{2C_-}\log \frac{C_+ \gamma}{\hat{C}\left[ \gamma^3+L^{-(d-1)} \right]}
\end{equation}
and $\hat{C}=\hat{C}(\beta_0,p_0,R,V_{\infty},d)$ is a positive constant. We assume that $L^{-(d-1)}\ll \gamma$. This guarantees $t_0$ to be positive, if $\gamma$ is sufficiently small. Moreover, we have $t_0 \to \infty$ as $\gamma \to +0$.

Now we prove this. First, Propositions~\ref{r+upperbound} and \ref{r-upperbound} imply
\begin{equation}
  0\leq r_{V}^{+}(t)+r_{V}^{-}(t)\leq \hat{C}\left[ \gamma^3 +L^{-(d-1)} \right]
\end{equation}
for some positive constant $\hat{C}=\hat{C}(\beta_0,p_0,R,V_{\infty},d)$. This and inequality~\eqref{estimate2} show that
\begin{align}
  \begin{aligned}
    \frac{d}{dt}[V_{\infty}-V(t)]
    & =-K_V(t)[V_{\infty}-V(t)]+r_{V}^{+}(t)+r_{V}^{-}(t) \\
    & \leq -C_+ \gamma e^{C_- t}+\hat{C}\left[ \gamma^3 +L^{-(d-1)} \right] \\
    & <0
  \end{aligned}
\end{align}
for $t\in [0,t_0]$. Hence $V(t)$ is increasing on $[0,t_0]$.

\subsection{Proof of a Refined Lower Bound of $V_{\infty}-V(t)$}\label{appendix3}
Finally, we give a proof of the refined lower bound~\eqref{estimate3_refined} of $V_{\infty}-V(t)$, under an additional assumption that $L^{-(d-1)}\ll \gamma$. As in~\cite{CMP1}, define $s_0$ by
\begin{equation}\label{s0def}
  s_0=s_0(t)=\min \left\{ s\in (0,t) \mathrel{\Big|} V(s)\geq \frac{V_0+\langle V \rangle_{s,t}}{2} \right\}.
\end{equation}
We prove below a lower bound
\begin{equation}\label{r+lowerbound}
  r_{V}^{+}(t)\geq C\bm{1}_{\{ t_1>t>\bar{t}/2 \}}\frac{\gamma^4}{t^{d+2}}.
\end{equation}
Here $\bar{t}=\bar{t}(\beta_0,p_0,R,V_{\infty},d)$ is a positive constant and $t_1$ is given by
\begin{equation}\label{t1}
  t_1=\bar{c}\gamma L^{d-1},
\end{equation}
where $\bar{c}=\bar{c}(\beta_0,p_0,R,V_{\infty},d)$ is a positive constant. To obtain this, we first prove the following: If we take $\bar{t}$ sufficiently large and $\bar{c}>0$ sufficiently small, (i)--(v) below hold.
\begin{enumerate}[label=(\roman*)]
  \item Bounds of $s_0$:
    \begin{equation}
      \frac{1}{C_-}\log \frac{3}{2}\leq s_0 \leq \frac{1}{C_+}\log 4
    \end{equation}
    for $t\geq \bar{t}/2$.
  \item A bound of $\langle V \rangle_{s,t}-V_0$:
    \begin{equation}
      \langle V \rangle_{s,t}-V_0 \geq \frac{2}{3}\gamma
    \end{equation}
    for $t\geq \bar{t}/2$ and $s\leq s_0$.
  \item A bound of $\langle V \rangle_{s_0,t}-\langle V \rangle_{0,t}$:
    \begin{equation}
      \langle V \rangle_{s_0,t}-\langle V \rangle_{0,t}\geq C\frac{\gamma}{t}
    \end{equation}
    for $t\geq \bar{t}/2$.
  \item A bound of $V(t)-\langle V \rangle_{s,t}$:
    \begin{equation}
      V(t)-\langle V \rangle_{s,t}\geq C\frac{\gamma}{t}
    \end{equation}
    for $t_1>t\geq \bar{t}/2$ and $s\leq s_0$, where $t_1$ is defined by equation~\eqref{t1}.
  \item Let $t_1>t\geq \bar{t}/2$ and $(x,\xi)\in I_{V}^{+}(t)$. Suppose that
    \begin{equation}
      \langle V \rangle_{0,t}<\xi_1<\langle V \rangle_{s_0,t},
    \end{equation}
    and let $\sigma_1$ be the largest $s\in (0,t)$ satisfying $\xi_1=\langle V \rangle_{s,t}$. Then we have $0<\sigma_1<s_0$. Moreover, if
    \begin{equation}
      |x_{\perp}-(t-\sigma_1)\xi_{\perp}|<R,
    \end{equation}
    then $\tau_1=\sigma_1$ and $\tau_2=0$, that is, $(x,\xi)$ produces exactly one pre-collision.
\end{enumerate}

We begin from the proof of (i). First, note that since $V(t)>V_0$ for $t>0$ (Section~\ref{appendix1}), we have
\begin{equation}
  V(0)=V_0<\frac{V_0+\langle V \rangle_{0,t}}{2}
\end{equation}
and
\begin{equation}
  V(t)>\frac{V_0+V(t)}{2}
\end{equation}
for $t>0$. These show that $s_0$ is well-defined. By definition~\eqref{s0def}, we have
\begin{equation}
  V(s_0)=\frac{V_0+\langle V \rangle_{s_0,t}}{2}.
\end{equation}
Use this to rewrite $V_{\infty}-V(s_0)$ as
\begin{equation}\label{eq:appendix2}
  V_{\infty}-V(s_0)=V_{\infty}-\frac{V_0+\langle V \rangle_{s_0,t}}{2}=\frac{\gamma}{2}+\frac{V_{\infty}-\langle V \rangle_{s_0,t}}{2}.
\end{equation}
On the other hand, by inequality~\eqref{estimate1}, we have
\begin{align}\label{eq:appendix3}
  \begin{aligned}
    V_{\infty}-V(s_0)
    & \leq \gamma e^{-C_+ s_0}+\gamma^3 \frac{A_+}{(1+s_0)^{d+2}}+L^{-(d-1)}\frac{B_+}{(1+s_0/L)^{d-1}} \\
    & \leq \gamma e^{-C_+ s_0}+\frac{\gamma}{4}    
  \end{aligned}
\end{align}
for $\gamma$ sufficiently small. We used $L^{-(d-1)}\ll \gamma$ here. Equality~\eqref{eq:appendix2}, inequality~\eqref{eq:appendix3} and $V(t)<V_{\infty}$ imply
\begin{equation}
  \frac{\gamma}{2}<\frac{\gamma}{2}+\frac{V_{\infty}-\langle V \rangle_{s_0,t}}{2}=V_{\infty}-V(s_0)\leq \gamma e^{-C_+ s_0}+\frac{\gamma}{4}.
\end{equation}
This gives an upper bound of $s_0$:
\begin{equation}\label{s0upperbound}
  s_0 \leq \frac{1}{C_+}\log 4.
\end{equation}
Next, by inequality~\eqref{estimate1}, we have
\begin{align}
  \begin{aligned}
    V_{\infty}-\langle V \rangle_{s,t}
    & =\frac{1}{t-s}\int_{s}^{t}(V_{\infty}-V(\sigma))\, d\sigma \\
    & \leq \frac{1}{t-s}\int_{s}^{t}\left[ \gamma e^{-C_+ \sigma}+\gamma^3 \frac{A_+}{(1+\sigma)^{d+2}}+L^{-(d-1)}\frac{B_+}{(1+\sigma/L)^{d-1}} \right] \, d\sigma \\
    & \leq \frac{C}{t-s}(\gamma +\gamma^3 A_+)+L^{-(d-1)}B_+.
  \end{aligned}
\end{align}
Taking $\gamma$ sufficiently small and $\bar{t}=\bar{t}(\beta_0,p_0,R,V_{\infty},d)$ sufficiently large, we have
\begin{equation}\label{eq:appendix4}
  V_{\infty}-\langle V \rangle_{s,t}\leq \frac{\gamma}{3}
\end{equation}
for $t\geq \bar{t}/2$ and $s\leq s_0$. Note that we used inequality~\eqref{s0upperbound} and $L^{-(d-1)}\ll \gamma$ here. Now inequality~\eqref{estimate2}, equality~\eqref{eq:appendix2} and inequality~\eqref{eq:appendix4} imply
\begin{equation}
  \gamma e^{-C_- s_0}\leq V_{\infty}-V(s_0)=\frac{\gamma}{2}+\frac{V_{\infty}-\langle V \rangle_{s_0,t}}{2}\leq \frac{2}{3}\gamma,
\end{equation}
which gives an lower bound of $s_0$:
\begin{equation}
  \frac{1}{C_-}\log \frac{3}{2}\leq s_0
\end{equation}
for $t\geq \bar{t}/2$.

Next, we prove (ii). Note that
\begin{equation}
  \langle V \rangle_{s,t}-V_0=\gamma -[V_{\infty}-\langle V \rangle_{s,t}].
\end{equation}
This and inequality~\eqref{eq:appendix4} imply
\begin{equation}
  \langle V \rangle_{s,t}-V_0 \geq \frac{2}{3}\gamma
\end{equation}
for $t\geq \bar{t}/2$ and $s\leq s_0$. This proves (ii).

For the proof of (iii), note that
\begin{align}\label{eq:appendix5}
  \begin{aligned}
    \langle V \rangle_{s_0,t}-\langle V \rangle_{0,t}
    & =\frac{1}{t-s_0}\int_{s_0}^{t}V(s)\, ds-\frac{1}{t}\int_{0}^{t}V(s)\, ds \\
    & =\left( \frac{1}{t-s_0}-\frac{1}{t} \right) \int_{0}^{t}V(s)\, ds-\frac{1}{t-s_0}\int_{0}^{s_0}V(s)\, ds \\
    & =\frac{s_0}{t-s_0}\left[ (V_{\infty}-\langle V \rangle_{0,s_0})-(V_{\infty}-\langle V \rangle_{0,t}) \right].
  \end{aligned}
\end{align}
By definition~\eqref{s0def} and $V(t)<V_{\infty}$, we have
\begin{equation}
  V(s)\leq \frac{V_0+\langle V \rangle_{s,t}}{2}<\frac{V_0+V_{\infty}}{2}
\end{equation}
for $s\leq s_0$. This implies
\begin{equation}\label{eq:appendix6}
  V_{\infty}-\langle V \rangle_{0,s_0}\geq \frac{\gamma}{2}.
\end{equation}
On the other hand, we have by inequality~\eqref{estimate1}
\begin{align}\label{eq:appendix7}
  \begin{aligned}
    V_{\infty}-\langle V \rangle_{0,t}
    & \leq \frac{1}{t}\int_{0}^{t}\left[ \gamma e^{-C_+ s}+\gamma ^3 \frac{A_+}{(1+s)^{d+2}}+L^{-(d-1)}\frac{B_+}{(1+s/L)^{d-1}} \right] \, ds\\
    & \leq \frac{C}{t}(\gamma +\gamma^3 A_+)+L^{-(d-1)}B_+ \\
    & \leq \frac{\gamma}{4}
  \end{aligned}
\end{align}
for $\gamma$ sufficiently small and $t\geq \bar{t}/2$ with $\bar{t}=\bar{t}(\beta_0,p_0,R,V_{\infty},d)$ sufficiently large. We used $L^{-(d-1)}\ll \gamma$ here. Now equation~\eqref{eq:appendix5}, inequalities~\eqref{eq:appendix6}, \eqref{eq:appendix7} and (i) imply
\begin{equation}
  \langle V \rangle_{s_0,t}-\langle V \rangle_{0,t}\geq C\frac{\gamma}{t}
\end{equation}
for $t\geq \bar{t}/2$. This proves (iii).

Next, we give a proof of (iv). By inequalities~\eqref{estimate1} and \eqref{estimate2}, taking $\bar{t}=\bar{t}(\beta_0,p_0,R,V_{\infty},d)$ sufficiently large, we have
\begin{align}
  \begin{aligned}
    V(t)-\langle V \rangle_{s,t}
    & \geq \frac{1}{t-s}\int_{s}^{t}\gamma e^{-C_- \sigma}\, d\sigma \\
    & \qquad -\gamma e^{-C_+ t}-\gamma^3 \frac{A_+}{(1+t)^{d+2}}-L^{-(d-1)}\frac{B_+}{(1+t/L)^{d-1}} \\
    & \geq \frac{1}{t-s_0}\int_{s_0}^{t}\gamma e^{-C_- \sigma}\, d\sigma \\
    & \qquad -\gamma e^{-C_+ t}-\gamma^3 \frac{A_+}{(1+t)^{d+2}}-L^{-(d-1)}\frac{B_+}{(1+t/L)^{d-1}} \\
    & \geq 2C_* \frac{\gamma}{t}-L^{-(d-1)}B_+
  \end{aligned}
\end{align}
for $t\geq \bar{t}/2$ and $s\leq s_0$, where $C_*=C_*(\beta_0,p_0,R,V_{\infty},d)$ is a positive constant. Now taking $\bar{c}=\bar{c}(\beta_0,p_0,R,V_{\infty},d)>0$ sufficiently small, we have
\begin{equation}
  V(t)-\langle V \rangle_{s,t}\geq 2C_* \frac{\gamma}{t}-L^{-(d-1)}B_+ \geq C_* \frac{\gamma}{t}
\end{equation}
for $t_1>t\geq \bar{t}/2$, where $t_1$ is defined by~\eqref{t1}. This proves (iv).

Finally, we prove (v). First, we show that
\begin{equation}\label{eq:appendix8}
  \langle V \rangle_{s_0,t}<\langle V \rangle_{s,t}
\end{equation}
for $t_1>t>s\geq \bar{t}/2$ with $\bar{t}=\bar{t}(\beta_0,p_0,R,V_{\infty},d)$ sufficiently large. To show this, note that by inequality~\eqref{estimate2}, we have for $t\geq \bar{t}/2$ with $\bar{t}$ sufficiently large
\begin{align}\label{eq:appendix9}
  \begin{aligned}
    \langle V \rangle_{s_0,t}
    & =V_{\infty}-\frac{1}{t-s_0}\int_{s_0}^{t}(V_{\infty}-V(s))\, ds \\
    & \leq V_{\infty}-\frac{1}{t-s_0}\int_{s_0}^{t}\gamma e^{-C_- s}\, ds \\
    & \leq V_{\infty}-2C_*\frac{\gamma}{t},
  \end{aligned}
\end{align}
where $C_*=C_*(\beta_0,p_0,R,V_{\infty},d)$ is a positive constant. On the other hand, we have by inequality~\eqref{estimate1}
\begin{align}
  \begin{aligned}
    \langle V \rangle_{s,t}
    & =V_{\infty}-\frac{1}{t-s}\int_{s}^{t}(V_{\infty}-V(\sigma))\, d\sigma \\
    & \geq V_{\infty}-\frac{1}{t-s}\int_{s}^{t}\left[ \gamma e^{-C_+ \sigma}+\gamma^3 \frac{A_+}{(1+\sigma)^{d+2}} \right. \\
    & \phantom{\geq V_{\infty}-\frac{1}{t-s}\int_{s}^{t}\gamma e^{-C+ \sigma}}\quad +\left. L^{-(d-1)}\frac{B_+}{(1+\sigma/L)^{d-1}} \right] \, d\sigma \\
    & \geq V_{\infty}-\gamma e^{-C_+ s}\frac{1-e^{-C_+(t-s)}}{C_+(t-s)} \\
    & \phantom{\geq V_{\infty}}\quad -\frac{\gamma^3 A_+}{(d+1)(t-s)(1+s)^{d+1}}-L^{-(d-1)}B_+.
  \end{aligned}
\end{align}
Considering separate cases of $s\leq t/2$ and $s\geq t/2$, we obtain
\begin{equation}\label{eq:appendix10}
  \langle V \rangle_{s,t}\geq V_{\infty}-C_*\frac{\gamma}{2t}
\end{equation}
for $t_1>t>s\geq \bar{t}/2$. Here we take $\bar{t}$ sufficiently large, $\gamma$ and $\bar{c}=\bar{c}(\beta_0,p_0,R,V_{\infty},d)$ sufficiently small. Inequalities~\eqref{eq:appendix9} and \eqref{eq:appendix10} imply
\begin{equation}
  \langle V \rangle_{s_0,t}\leq V_{\infty}-2C_*\frac{\gamma}{t}<V_{\infty}-C_*\frac{\gamma}{t}\leq \langle V \rangle_{s,t}
\end{equation}
for $t_1>t>s\geq \bar{t}/2$, which proves inequality~\eqref{eq:appendix8}. Now let $t_1>t\geq \bar{t}/2$ and $(x,\xi)\in I_{V}^{+}$. Suppose that
\begin{equation}
  \langle V \rangle_{0,t}<\xi_1<\langle V \rangle_{s_0,t},
\end{equation}
and let $\sigma_1$ be the largest $s\in (0,t)$ satisfying $\xi_1=\langle V \rangle_{s,t}$. We prove that $0<\sigma_1<s_0$. Since $V(t)$ is increasing on $[0,t_0]$ (Section~\ref{appendix2}), we have
\begin{equation}
  \xi_1=\langle V \rangle_{\sigma_1,t}<\langle V \rangle_{s_0,t}\leq \langle V \rangle_{s,t}
\end{equation}
for $s_0 \leq s\leq t_0$. This implies that $\sigma_1<s_0$ or $\sigma_1>t_0$. Suppose that $\sigma_1>t_0$. By definition~\eqref{t0} of $t_0$ and $L^{-(d-1)}\ll \gamma$, we have $\sigma_1>t_0 \geq \bar{t}/2$ for $\gamma$ sufficiently small. So we can apply inequality~\eqref{eq:appendix8} with $s=\sigma_1$: We have
\begin{equation}
  \langle V \rangle_{s_0,t}<\langle V \rangle_{\sigma_1,t}=\xi_1.
\end{equation}
This contradicts the assumption that $\xi_1<\langle V \rangle_{s_0,t}$. Therefore, we conclude that $\sigma_1<s_0$. If
\begin{equation}
  |x_{\perp}-(t-\sigma_1)\xi_{\perp}|<R,
\end{equation}
it is easy to see that $\tau_1=\sigma_1$. The reflected horizontal velocity $\xi_{1}'(\tau_1)$ at time $\tau_1$ is
\begin{equation}
  \xi_{1}'(\tau_1)=2V(\tau_1)-\xi_1.
\end{equation}
Since $\tau_1=\sigma_1<s_0$, we have
\begin{equation}
  2V(\tau_1)<V_0+\langle V \rangle_{\tau_1,t}=V_0+\xi_1
\end{equation}
by definition~\eqref{s0def} of $s_0$. These imply
\begin{equation}
  \xi_{1}'(\tau_1)<V_0.
\end{equation}
Since $V(t)>V_0$ for $t>0$, no further pre-collisions occur: $\tau_2=0$. This proves (v).

Now we give the lower bound~\eqref{r+lowerbound} of $r_{V}^{+}(t)$. First, note that by (iii) and (iv), we have
\begin{equation}
  \langle V \rangle_{0,t}<\langle V \rangle_{s_0,t}<V(t)
\end{equation}
for $t_1>t\geq \bar{t}/2$. So if we define $I(t)$ as
\begin{equation}
  I(t)=\int_{C_{V}^{+}(t)}\, dS\int \, d\xi_{\perp}\int_{\langle V \rangle_{0,t}}^{\langle V \rangle_{s_0,t}}(\xi_1-V(t))^2 \left( e^{-\beta_0 |\xi_0|^2}-e^{-\beta_0 |\xi|^2} \right) \, d\xi_1,
\end{equation}
we have $r_{V}^{+}(t)\geq \bar{C}I(t)$ for $t_1>t\geq \bar{t}/2$. See definition~\eqref{r+} of $r_{V}^{+}(t)$. Next, take $x\in C_{V}^{+}(t)$ satisfying $|x_{\perp}|<R/2$; take $\xi$ satisfying $\langle V \rangle_{0,t}<\xi_1<\langle V \rangle_{s_0,t}$ and $|\xi_{\perp}|>R/(2t)$; and let $\sigma_1$ be the largest $s\in (0,t)$ satisfying $\xi_1=\langle V \rangle_{s,t}$. These imply
\begin{equation}
  |x_{\perp}-(t-\sigma_1)\xi_{\perp}|\leq |x_{\perp}|+t|\xi_{\perp}|<R.
\end{equation}
Therefore, by (ii), (v), $V(t)>V_0$, and definition~\eqref{s0def} of $s_0$, we have
\begin{align}
  \begin{aligned}
    \xi_{1}^{2}-\xi_{01}^{2}
    & =\xi_{1}^{2}-(2V(\tau_1)-\xi_1)^2 \\
    & =4V(\tau_1)(\langle V \rangle_{\tau_1,t}-V(\tau_1)) \\
    & >2V(\tau_1)(\langle V \rangle_{\tau_1,t}-V_0) \\
    & >\frac{4}{3}V_0 \gamma
  \end{aligned}
\end{align}
for $t_1>t\geq \bar{t}/2$. So we have
\begin{align}
  \begin{aligned}
    I(t)
    & \geq C\gamma \int_{|\xi_{\perp}|<\frac{R}{2t}}e^{-\beta_0 |\xi_{\perp}^2|}\, d\xi_{\perp}\int_{\langle V \rangle_{0,t}}^{\langle V \rangle_{s_0,t}}(\xi_1-V(t))^2 \, d\xi_1 \\
    & \geq C\frac{\gamma}{t^{d-1}}\left[ (\langle V \rangle_{s_0,t}-V(t))^3 -(\langle V \rangle_{0,t}-V(t))^3 \right] \\
    & =C\frac{\gamma}{t^{d-1}}(\langle V \rangle_{s,t}-V(t))^2 (\langle V \rangle_{s_0,t}-\langle V \rangle_{0,t})
  \end{aligned}
\end{align}
for $t_1>t\geq \bar{t}/2$ and some $s\in (0,s_0)$. By (iii) and (iv), we have
\begin{equation}
  r_{V}^{+}(t)\geq \bar{C}I(t)\geq C\frac{\gamma^4}{t^{d+2}}
\end{equation}
for $t_1>t\geq \bar{t}/2$. This proves inequality~\eqref{r+lowerbound}.

Finally, by Proposition~\ref{r-nonnegative} and inequality~\eqref{r+lowerbound}, we have
\begin{align}
  \begin{aligned}
    V_{\infty}-V(t)
    & =\gamma e^{-\int_{0}^{t}K_V(\sigma)\, d\sigma}+\int_{0}^{t}e^{-\int_{s}^{t}K_V(\sigma)\, d\sigma}\left[ r_{V}^{+}(s)+r_{V}^{-}(s) \right] \, ds \\
    & \geq \gamma e^{-C_- t}+\int_{0}^{t}e^{-C_-(t-s)}r_{V}^{+}(s)\, ds \\
    & \geq \gamma e^{-C_- t}+C\gamma^4 \int_{\bar{t}/2}^{t}\frac{e^{-C_-(t-s)}}{s^{d+2}}\, ds 
  \end{aligned}
\end{align}
for $t_1>t\geq \bar{t}/2$. For $t_1>t>\bar{t}$, we have
\begin{equation}
  \int_{\bar{t}/2}^{t}\frac{e^{-C_-(t-s)}}{s^{d+2}}\, ds\geq \frac{1-e^{-C_- \bar{t}/2}}{C_- t^{d+2}}.
\end{equation}
Therefore, by choosing a positive constant $A_-=A_-(\beta_0,p_0,R,V_{\infty},d)$ appropriately, we have
\begin{equation}
  V_{\infty}-V(t)\geq \gamma e^{C_- t}+\gamma^4 \frac{A_-}{t^{d+2}}
\end{equation}
for $t_1>t>\bar{t}$. If $t>L$ in addition, a similar argument as in Section~\ref{sec:proof2} shows
\begin{equation}
  V_{\infty}-V(t)\geq \gamma e^{C_- t}+\gamma^4 \frac{A_-}{t^{d+2}}+\frac{B_-}{t^{d-1}}.
\end{equation}
This proves inequality~\eqref{estimate3_refined}.



\medskip
Received March 2017; revised xxxx 20xx.
\medskip


\begin{thebibliography}{99}
  \bibitem{ACMP1} (MR2405148) [10.1051/m2an:2008007]
    \newblock K. Aoki, G. Cavallaro, C. Marchioro and M. Pulvirenti,
    \newblock \emph{On the motion of a body in thermal equilibrium immersed in a perfect gas},
    \newblock M2AN Math. Model. Numer. Anal., \textbf{42} (2008), 263--275.

  \bibitem{BJJ1} (MR1863781)
    \newblock A. Belmonte, J. Jacobsen and A. Jayaraman,
    \newblock \emph{Monotone solutions of a nonautonomous differential equation for a sedimenting sphere},
  \newblock Electron. J. Differential Equations, \textbf{62} (2001), 1--17.
    
  \bibitem{BCM1} (MR3308141) [10.1007/978-3-319-14759-8]
    \newblock P. Butt\`{a}, G. Cavallaro and C. Marchioro,
    \newblock \emph{Mathematical Models of Viscous Friction},
    \newblock Lecture Notes in Mathematics, \textbf{2135}, Springer, Cham, 2015.

  \bibitem{CCM1} (MR2353147) [10.1142/S0218202507002315]
    \newblock S. Caprino, G. Cavallaro and C. Marchioro,
    \newblock \emph{On a microscopic model of viscous friction},
    \newblock Math. Models Methods Appl. Sci., \textbf{17} (2007), 1369--1403.

  \bibitem{CMP1} (MR2212220) [10.1007/s00220-006-1542-7]
    \newblock S. Caprino, C. Marchioro and M. Pulvirenti,
    \newblock \emph{Approach to equilibrium in a microscopic model of friction},
    \newblock Comm. Math. Phys., \textbf{264} (2006), 167--189.
    
  \bibitem{Cavallaro1} (MR2361025)
    \newblock G. Cavallaro,
    \newblock \emph{On the motion of a convex body interacting with a perfect gas in the mean-field approximation},
    \newblock Rend. Mat. Appl., \textbf{27} (2007), 123--145.

  \bibitem{CM1} (MR2740712) [10.1142/S0218202510004854]
    \newblock G. Cavallaro and C. Marchioro,
    \newblock \emph{On the approach to equilibrium for a pendulum immersed in a Stokes fluid},
    \newblock Math. Models Methods Appl. Sci., \textbf{20} (2010), 1999--2019.

  \bibitem{CMT1} (MR2855358) [10.1007/s11565-011-0127-3]
    \newblock G. Cavallaro, C. Marchioro and T. Tsuji,
    \newblock \emph{Approach to equilibrium of a rotating sphere in a Stokes flow},
    \newblock Ann. Univ. Ferrara Sez. VII Sci. Mat., \textbf{57} (2011), 211--228.
    
  \bibitem{CS1} (MR3158809) [10.1007/s00205-013-0675-z]
    \newblock X. Chen and W. Strauss,
    \newblock \emph{Approach to equilibrium of a body colliding specularly and diffusely with a sea of particles},
    \newblock Arch. Ration. Mech. Anal., \textbf{211} (2014), 879--910.
    
  \bibitem{CS2} (MR3359707) [10.1007/s00220-015-2368-y]
    \newblock X. Chen and W. Strauss,
    \newblock \emph{Velocity reversal criterion of a body immersed in a sea of particles},
    \newblock Comm. Math. Phys., \textbf{338} (2015), 139--168.

  \bibitem{FSS1} (MR3471075) [10.1063/1.4943013]
    \newblock C. Fanelli, F. Sisti and G. V. Stagno,
    \newblock \emph{Time dependent friction in a free gas},
    \newblock J. Math. Phys., \textbf{57} (2016), 1--12.

  \bibitem{RS1} (MR3274363) [10.1137/140954003]
    \newblock C. Ricciuti and F. Sisti,
    \newblock \emph{Effects of concavity on the motion of a body immersed in a Vlasov gas},
    \newblock SIAM J. Math. Anal., \textbf{46} (2014), 3579--3611.

\end{thebibliography}
\end{document}